\documentclass[a4paper,12pt]{amsart}
\usepackage[utf8]{inputenc}
\usepackage{amsmath}
\usepackage{amsfonts}
\usepackage{amsthm}
\usepackage{mathtools}
\usepackage{amssymb}
\usepackage[mathscr]{euscript}
\usepackage{xcolor}
\usepackage{enumerate}
\usepackage{tikz}
\usepackage{tikz}
\usepackage[utf8]{inputenc}

\usepackage[
margin=1in,
marginpar=2cm,
includefoot,
footskip=30pt,
]{geometry}

\newtheorem{theorem}[equation]{Theorem}
\newtheorem{lemma}[equation]{Lemma}
\newtheorem{proposition}[equation]{Proposition}

\theoremstyle{definition}
\newtheorem{definition}[equation]{Definition}

\theoremstyle{remark}
\newtheorem{example}[equation]{Example}
\newtheorem{remark}[equation]{Remark}

\numberwithin{equation}{section}

\newcommand{\I}{{I}}
\newcommand{\bI}{\mathbb{I}}

\newcommand{\af}{\alpha}

\newcommand{\mdl}[3]{_{#2}#1_{#3}}
\newcommand{\tG}{G}

\newcommand{\mcD}{\mathcal{D}}
\newcommand{\m}{{}^{-1}}

\newcommand{\id}{{\rm id}}

\newcommand{\mcI}{{\mathcal{I}}}

\newcommand{\op}[1]{#1^{\rm op}}
\newcommand{\End}[1]{\operatorname{End}(#1)}
\newcommand{\opcop}[1]{#1^{\rm opcop}}
\newcommand{\copr}[2]{{#1}_{#2}}

\title[Partial Hopf actions]{Partial Hopf actions  on generalized matrix algebras}

\author[D. Bagio]{Dirceu Bagio}
\author[E. Batista]{Eliezer Batista}
\author[H. Pinedo]{H. Pinedo}
\address[Dirceu Bagio and Eliezer Batista]{Departamento de Matem\'atica, Universidade Federal de Santa Catarina, 88040-970 Florian\'opolis SC, Brazil. }
\email{d.bagio@ufsc.br \\ e.batista@ufsc.br }

\address[Hector Pinedo]{Escuela de Matematicas, Universidad Industrial de Santander, Cra. 27 Calle 9, UIS Edificio 45, Bucaramanga, Colombia.}
\email{hpinedot@uis.edu.co}

\begin{document}
	
	\thanks{D. Bagio was partially supported by CAPES-PRINT 88887.894056/2023-00. H. Pinedo was partially supported by CAPES-PRINT 88887.895167/2023-00 and by FAPESP, process n°: 2023/14066-5.}

	\begin{abstract}  
		Let $\Bbbk$ be a field, $H$ a Hopf algebra over $\Bbbk$, and $R = (_iM_j)_{1 \leq i,j \leq n}$ a ge\-ne\-rali\-zed matrix algebra. In this work, we establish necessary and sufficient conditions for $H$ to act partially on $R$. To achieve this, we introduce the concept of an opposite covariant pair and demonstrate that it satisfies a universal property. In the special case where $H = \Bbbk G$ is the group algebra of a group $G$, we recover the conditions given in \cite{BP} for the existence of a unital partial action of $G$ on $R$.
	\end{abstract}
	
	\maketitle

	\section{Introduction}
	
	\vspace{.1cm}
	
	Generalized matrix algebras were first introduced in the literature by Brown in \cite{Brown}, where the author explored matrix rings associated with orthogonal groups. These matrix rings form a large class of rings that arise in various areas of algebra and module theory and have been extensively studied. Broadly speaking, a generalized matrix algebra is an algebra formed as a direct sum of $\Bbbk$-vector spaces, 
	\[
	R = \bigoplus\limits_{i,j=1}^n {}_i M_j,
	\]
	where the diagonal components $R_i = {}_i M_i$ are $\Bbbk$-algebras, and the off-diagonal components ${}_i M_j$ are $(R_i, R_j)$-bimodules. The algebra structure of $R$ emulates the usual structure of a matrix algebra by employing a series of bimodule morphisms to define the product between entries residing solely in the bimodules.
	\vspace{.1cm}

	Notably, the class of $2 \times 2$ generalized matrix algebras has been shown to correspond one-to-one with Morita contexts \cite{Morita}. The study of the so-called \emph{Morita ring} has gained significant attention in recent years, as it encapsulates key algebraic properties of the Morita context itself \cite{HV, KT, Tang}. Another context in which generalized matrix algebras arise is in the study of idempotents. Specifically, any algebra $A$ equipped with a complete set of orthogonal idempotents $\{e_i\}_{i=1}^n$ can be decomposed using Peirce decomposition. This decomposition can be organized into an $n \times n$ generalized matrix algebra $A^\pi$, which is isomorphic to $A$. For further details, see, for instance, \cite{ABW}.
	\vspace{.1cm}
	
	Partial actions of Hopf algebras were introduced in \cite{CJ}, motivated by the Galois theory of partial group actions studied in \cite{DFP}. Since then, numerous works on partial actions of Hopf algebras have been developed, exploring several problems that were initially addressed from a global perspective through a partial framework. For a detailed account of this topic, the interested reader is referred to \cite{B}.  
	\vspace{.1cm}

	The main objective of this work is to establish a connection between two significant topics: partial actions of Hopf algebras and generalized matrix algebras. Specifically, we aim to determine the necessary and sufficient conditions for a Hopf algebra to act partially on a generalized matrix algebra. The analogous problem was considered in \cite{BP} for the case of partial group actions.
	\vspace{.1cm}

	This paper is structured as follows. In Section 2, we review the fundamental concepts of generalized matrix algebras, partial actions, partial representations of Hopf algebras and partial smash products. We also introduce the notion of an opposite covariant pair and demonstrate that it satisfies a universal property. In Section 3, we prove the main result of this work, Theorem 3.4, which provides the necessary and sufficient conditions for a Hopf algebra $H$ to act partially on a generalized matrix algebra $R$. Conditions for the existence of a partial action of a group $\tG$ on $R$ were previously discussed in \cite{BP}. The relationship between these results and those of Theorem 3.4 is analyzed in Section 4, focusing on the specific case where the Hopf algebra is the group algebra $\Bbbk \tG$. Finally, in Section 5, we discuss Morita-equivalent partial actions of a Hopf algebra. Additionally, we address an inconsistency in Proposition 73 of \cite{AF}.

	\subsection*{Conventions}\label{subsec:conv}
	Throughout this work, $\Bbbk$ is a field, by an algebra  we mean an associative  unital $\Bbbk$-algebra, and vector spaces and tensor product are over $\Bbbk$. Given an algebra $A$, the opposite algebra $\op{A}$ is the vector space $A$ with the multiplication opposite, i.e., $a\cdot_{{\rm op}} b=ba$, for all $a,b\in A$. Given a Hopf algebra $H$ and $h\in H$, we use the simplified Sweedler's notation $\Delta(h)=h_{(1)}\otimes h_{(2)}$ to denote the comultiplication of $h$. We will assume that every Hopf algebra $H$ has bijective antipode. For a coalgebra $C$, the coopposite coalgebra $C^{\rm cop}$ is the vector space $C$ with comultiplication $\Delta_{\rm cop}(c)=c_{(2)}\otimes c_{(1)}$, where $\Delta(c)=c_{(1)}\otimes c_{(2)}$. In order to avoid confusion, we will denote $\Delta_{\rm cop}(c)=c_{[1]}\otimes c_{[2]}$, for all $c\in C^{\rm cop}$. Finally, the set $\{1,\ldots ,n\}$ will be denoted by $\bI_n.$

	\section{Preliminaries}
	In this section, we review definitions and results that will be used in the rest of the work. We introduce the concept of opposite covariant pair and prove that it satisfies a certain universal property.
	
	\subsection{Generalized matrix algebras}
	We begin with the following definition, which will allow us to introduce one of the key concepts of this work, namely, the notion of a generalized matrix ring.
	\begin{definition}
		Let $n$ be a positive integer. A {\it generalized matrix datum} of order $n$ over $\Bbbk$ is a triple
		\[
		\mathcal{R}= \left( \{ {}_i M_j \}_{i,j\in\bI_n} , \{ \theta_{ijk} \}_{ i,j,k \in \bI_n} ,\{ \eta_i \}_{i \in \bI_n} \right) ,
		\]
		such that: for every $i,j,k\in \I_n$,
		\begin{enumerate}
			\item[\footnotesize{(GMD1)}] ${}_i M_j$ is a $\Bbbk$-vector space, \vspace{.05cm}
			\item[\footnotesize{(GMD2)}] $\theta_{ijk} :{}_i M_j \otimes {}_j M_k \rightarrow {}_i M_k$ and $\eta_i :\Bbbk \rightarrow {}_i M_i$  are $\Bbbk$-linear map, \vspace{.05cm}
			\item[\footnotesize{(GMD3)}] $\theta_{ijk}$ and $\eta_i$ satisfy the following compatibility identities:
			\begin{equation*}
				\theta_{ikl} \circ \left( \theta_{ijk} \otimes \id_{{}_kM_l} \right) =\theta_{ijl} \circ \left( \id_{{}_iM_j} \otimes\theta_{jkl}  \right) 
			\end{equation*}
			and
			\begin{equation*}
				\theta_{iij} \circ (\eta_i \otimes \id_{{}_i M_j}) =\id_{{}_i M_j} =\theta_{ijj} \circ (\id_{{}_i M_j} \otimes \eta_j ) .
			\end{equation*}
		\end{enumerate}    
	\end{definition}
	
	From the previous identities, one can easily deduce the following result.
	
	\begin{proposition}\label{prop-def-gma}
		Let $R$ be a generalized matrix datum of order $n$ over $\Bbbk$. Then, for every $i,j,k\in \I_n$, we have that:
		\begin{enumerate} [\rm (i)]
			\item the vector spaces ${}_i M_i$ are unital $\Bbbk$-algebras, henceforth de\-no\-ted as $R_i$,\vspace{.1cm}
			\item the vector spaces ${}_i M_j$ are $(R_i , R_j)$ bimodules,\vspace{.1cm}
			\item the linear maps $\theta_{ijk}$ are morphisms of $(R_i ,R_k)$-bimodules and are balanced over $R_j$, that is, $\theta_{ijk} :{}_i M_j \otimes_{R_j} {}_j M_k \rightarrow {}_i M_k$.
		\end{enumerate}\qed
	\end{proposition}
	From Proposition \ref{prop-def-gma}, we conclude that: for all $i,j\in \I_n$,
	\begin{itemize}
		\item [$\diamond$] $\,\,\theta_{iii}:\!_i{M}_{i}\otimes_{R_i}\,\!\!\mdl{M}{i}{i}\to \!\mdl{M}{i}{i}$ is simply the multiplication in $R_i$,\vspace{.1cm}
		\item [$\diamond$] $\,\,\theta_{iij}:\!_i{M}_{i}\otimes_{R_i}\,\!\!\mdl{M}{i}{j}\to \!\mdl{M}{i}{j}$, is the left $R_i$-module structure on ${}_iM_j$,\vspace{.1cm}
		\item [$\diamond$] $\,\,\theta_{ijj}:\!_i{M}_{j}\otimes_{R_j}\,\!\!\mdl{M}{j}{j}\to \!\mdl{M}{i}{j}$, is the right $R_j$-module structure on ${}_iM_j$.
	\end{itemize}
	Given a generalized matrix datum, one can associate to the vector space  \[
	R=\bigoplus_{i,j=1}^n {}_iM_j\]
	a algebra structure  as follows. For all $i,j,k,l \in \bI_n$, define the linear map $\mu_{ij,kl}: {}_i M_j \otimes {}_k M_l \rightarrow {}_i M_l$
	by
	\[
	\mu_{ij,kl}=\left\{ \begin{array}{ll} \theta_{ijl}, & \text{ if } j=k ,\\
		0 ,& \text{ otherwise}.\end{array}\right. 
	\]
	Hence, we have linear maps $\iota_{il}\circ\mu_{ij,kl}: {}_i M_j \otimes {}_k M_l \to R$, where $\imath_{il} : {}_iM_l \rightarrow R$ is the inclusion map. Using the universal property of the direct sum
	\[
	R\otimes R \cong \bigoplus_{i,j,k,l =1}^n {}_iM_j \otimes {}_kM_l,
	\]
	we obtain a linear map $\mu : R\otimes R \rightarrow R$. We also define the linear map $\eta:\Bbbk \rightarrow R$ by
	\[
	\eta =\sum_{i=1}^n \imath_{ii} \circ \eta_i.\]
	\begin{proposition}\label{rdata}
		Given a generalized matrix datum
		\[
		\mathcal{R}= \left( \{ {}_i M_j \}_{i,j\in\bI_n} , \{ \theta_{ijk} \}_{ i,j,k \in \bI_n} ,\{ \eta_i \}_{i \in \bI_n} \right) ,
		\]
		the triple $(R, \mu , \eta )$ defines a unital algebra over the field $\Bbbk$. 
	\end{proposition}
	
	\begin{proof}
		One can write a generalized matrix algebra of order $n$ over $\Bbbk$ as a square array
		\begin{align}\label{def-gmr}	R = \left( \begin{matrix} R_1 & \mdl{M}{1}{2}& \ldots & \mdl{M}{1}{n} \\ 
				\mdl{M}{2}{1} & R_2 & \ddots & \vdots  \\ 
				\vdots & \ddots & \ddots & \mdl{M}{(n-1)}{n} \\ 
				\mdl{M}{n}{1} &\ldots & \mdl{M}{n}{(n-1)} & R_n \end{matrix} \right) .
		\end{align} 
		A typical element of $R$ can be  written as a  matrix $x=(x_{ij})$. Considering the vector space structure on the direct sum $R$ one can see that the addition is done componentwise and the multiplication $\mu$ on $R$ is the row-column matrix multiplication,
		\[
		\mu (\left( x_{ij} \right) \otimes \left( y_{ij} \right)) = \left( x_{ij} \right) \left( y_{ij} \right)=
		\left( \sum_{k=1}^n x_{ik} y_{kj}\right) =\left( \sum_{k=1}^n \theta_{ikj} (x_{ik} \otimes  y_{kj}) \right) .
		\]
		The associativity relations of the maps $\theta_{ijk}$ given in (GMD3) imply that the matrix multiplication is associative, and the linear map $\eta$ is indeed a unit map. Explicitly, 
		\[
		\eta (\lambda )= \lambda \left( \begin{array}{cccc} 
			1_{1} & 0 & \cdots & 0 \\ 0 & 1_{2} & \cdots & 0 \\
			\vdots & \vdots & \ddots & \vdots \\
			0 & 0 & \cdots & 1_{n} \end{array} \right)  , \quad \lambda\in \Bbbk,
		\]
		in which $1_i$ denotes the unit of the corresponding algebra $R_i$.
	\end{proof}
	
	\begin{definition}
		The algebra constructed in Proposition  \ref{rdata} will be called the \emph {generalized matrix algebra} defined by the generalized matrix datum $\mathcal{R}$.    
	\end{definition}

	In what follows, a generalized matrix algebra $R$ as defined above will be denoted by $R=(\mdl {M}{i}{j})_{i,j\in \bI_n}$. Also, a diagonal matrix in $R$ will be denoted by $r=\operatorname{diag}(r_{11},\ldots,r_{nn})$, that is, $r=(r_{ij})$ is the matrix such that $r_{ij}=0$ if $i\neq j$.  
	
	\begin{remark}\label{bimodule-structure}
		Let $R=(\mdl {M}{i}{j})_{i,j\in \bI_n}$ be a generalized matrix algebra and $k,l\in \bI_n$. Notice that $R$ is a left $R_k$-module with the following structure: for all $r_k\in R_k$ and $(m_{ij})\in R$,
		\[r_k\cdot (m_{ij})=(\tilde{m}_{ij}),\quad\text{ where }\, \tilde{m}_{ij}=\begin{cases}
			m_{ij},& \text{if } i\neq k,\\
			r_k\cdot m_{ij},& \text{if } i=k.\\
		\end{cases}\]
		Similarly, $R$ is a right $R_l$-module via: for all $r_l\in R_l$ and $(m_{ij})\in R$,
		\[(m_{ij})\cdot r_l=(\tilde{m}_{ij}),\quad\text{ where }\, \tilde{m}_{ij}=\begin{cases}
			m_{ij},& \text{if } j\neq l,\\
			m_{ij}\cdot r_l,& \text{if } j=l.\\
		\end{cases}\]
		%Verifying that $R$ is an $(R_k,R_l)$-bimodule with the previous structure is straightforward.
	\end{remark}

	\begin{example}\label{linc}
		Let $\mathcal{C}$ be a $\Bbbk$-linear category with finitely many objects. For each pair of objects $X_i , X_j \in \mathcal{C}^{(0)}$, denote by ${}_i M_j =\text{Hom}_{\mathcal{C}}(X_j , X_i)$, which is a vector space. The total space
		\[
		A(\mathcal{C}) =\bigoplus_{i,\,j} {}_i M_j
		\]
		is a generalized matrix algebra with the multiplication given by the ordinary composition of morphisms in the category $\mathcal{C}$.
	\end{example}
	
	\begin{example}
		Let $(A,B,M,N,\mu,\nu)$ be a Morita context between the algebras $A$ and $B$. Then 
		$$ R = \left( \begin{matrix} A & M \\ N & B \end{matrix} \right)$$
		is a $2$-by-$2$  generalized matrix algebra which is called a {\it Morita ring}.
		Conversely, each $2$-by-$2$  generalized matrix algebra determines a Morita context \cite{HV}. 
	\end{example}
	
	\begin{example}
		Let $A$ be an algebra and $e\in A$ be an idempotent element. One can construct the following generalized matrix algebra:
		\[
		R=\left( \begin{array}{cc} eAe & eA(1-e) \\ (1-e)Ae & (1-e)A(1-e) \end{array} \right) .
		\]
	\end{example}

	\subsection{Partial actions and partial representations of Hopf algebras} Left partial actions of bialgebras on algebras were introduced in \cite{CJ}, since this notion does not deal with symmetric conditions we will give it as it appears in  \cite{AB}.
	\begin{definition}
		{\rm   {\it A symmetric  left partial action }of a Hopf algebra $H$ over a $\Bbbk$-algebra $A$ is a linear map 
			\[
			\begin{array}{rccl} \cdot: & H\otimes A & \rightarrow & A \\
				\, & h\otimes a & \mapsto & h\cdot a \end{array}
			\]
			satisfying: for all $h,k\in H$ and $a,b\in A$,
			\begin{enumerate}
				\item[\footnotesize{(LPA1)}]$\,$ $1_H \cdot a =a$, 
				\item[\footnotesize{(LPA2)}]$\,$ $h\cdot (ab)=(h_{(1)} \cdot a)(h_{(2)}\cdot b)$, 
				\item[\footnotesize{(LPA3)}]$\,$ $h\cdot (k\cdot a)=(h_{(1)}\cdot 1_A )(h_{(2)}k\cdot a)=(h_{(1)}k\cdot a)(h_{(2)}\cdot 1_A) $.
		\end{enumerate}}
	\end{definition}
	\noindent It follows from (LPA2) and (LPA3) that: for all $h,k\in H$ and $a,b\in A$,
	\begin{align}
		\label{prodc}
		&   \left(h_{(1)}\cdot (k\cdot a)\right)(h_{(2)}\cdot b)=(h_{(1)}k\cdot a)(h_{(2)}\cdot 1_A)(h_{(3)}\cdot b)=
		(h_{(1)}k\cdot a)(h_{(2)}\cdot b),\\   
		\label{prodc-1}
		&(h_{(1)}\cdot a)\left(h_{(2)}\cdot (k\cdot b)\right)=(h_{(1)}\cdot a)(h_{(2)}\cdot 1_A)\left(h_{(3)}k\cdot b\right)=(h_{(1)}\cdot a)(h_{(2)}k\cdot b).
	\end{align}

	Similarly, one can give the notion of a right partial action of a Hopf algebra as follows.
	
	\begin{definition}
		{\rm
			{\it    A symmetric right partial action }of a Hopf algebra $H$ over a $\Bbbk$-algebra $A$ is a linear map 
			\[
			\begin{array}{rccl} \cdot: & A\otimes H & \rightarrow & A \\
				\, & a\otimes h & \mapsto & a\cdot h \end{array}
			\]
			satisfying: for all $h,k\in H$ and $a,b\in A$,
			\begin{enumerate}
				\item[\footnotesize{(RPA1)}]$\,$ $a\cdot 1_H =a$,
				\item[\footnotesize{(RPA2)}]$\,$ $(ab)\cdot h=(a\cdot h_{(1)} )(b\cdot h_{(2)})$, 
				\item[\footnotesize{(RPA3)}]$\,$ $(a\cdot k)\cdot h=(a\cdot kh_{(1)})(1_A \cdot h_{(2)})=(1_A\cdot h_{(1)})(a\cdot kh_{(2)})$.
		\end{enumerate}}
	\end{definition}
	\begin{remark}
		From now on in this work by left (right) partial action we mean a symmetric left (right) partial action. Also, we say that an algebra $A$ is a left (right)  partial $H$-module algebra if it is endowed with a left (right) partial action. %we made the same consideration in the case of right partial action. Also by 
	\end{remark}
	\begin{example} \label{ex-pa-group}
		Given a unital partial action     $ \alpha =\left(  A_g, \alpha_g  \right)_{g\in \tG}$ of a group $\tG$ on a unital algebra $A$ (see Definition 1.1 in \cite{DE}) in which each ideal $A_g \trianglelefteq A$ is generated by a central idempotent element $1_g \in A$, one can define a left partial action of the group algebra $\Bbbk G$ on $A$ by %linearly extending the expression
		\[
		g \cdot a= \alpha_g (a1_{g^{-1}}),\,\, a\in A,\,\, g\in G.
		\] 
		Conversely, let $H=\Bbbk \tG$  and $\cdot: H\otimes A\to A$ a left partial action of $H$ on $A$. For each $g\in \tG$, consider $1_g:=g\cdot 1_A$. Notice that $1_g$ is a central idempotent of $A$. Indeed, $1_g1_g=(g\cdot 1_A)(g\cdot 1_A)=g\cdot (1_A1_A)=g\cdot 1_A=1_g$. Also, for all $a\in A$,
		\begin{align*}
			1_ga&=(g\cdot 1_A)(1_{\Bbbk\tG}\cdot a)=(g\cdot 1_A)(gg\m\cdot a)\\
			&=g\cdot (g\m\cdot a)=(gg\m \cdot a)(g\cdot 1_A) \qquad\quad \big(\textrm{by (LPA3)} \big)\\
			&=a1_g.
		\end{align*}
		Then $\alpha=\big(D_g,\af_g\big)_{g\in \tG}$ is a unital partial action of $\tG$ on $A$, where $D_g=A1_g, 
		$ and $\af_g(a)=g\cdot a, a\in D_{g\m}.$
		
	\end{example}
	
	%\begin{example} If $B$ is a left $H$ module algebra via $\triangleright$ and $e=e^2 \in B$ is a central idempotent element, then one can define a left partial action of $H$ on $A=eB$ by  \[  h\cdot (eb)=e(h\triangleright (eb))  \]
	%\end{example}
	\begin{example}\label{ltor}
		Let $\cdot:H\otimes A\to A$ be a left partial action of a Hopf algebra $H$ over an algebra $A$. Consider the linear map
		$\triangleleft: A\otimes \op{H}\to  A $ defined by 
		\begin{equation}\label{lartion}
			a\triangleleft  h=h\cdot a,\quad a\in {A},\,\,h\in \op{H}.
		\end{equation} It is straightforward to verify that $\triangleleft$ is a right  partial action of $\op{H}$ over $ A$. 
	\end{example}

	%Moreover, the same action induces a right partial action of $\opcop{H}$ on $\op A,$ again  observe that (RPA1) is immediate. For(RPA2), consider $a,b\in \op{A}$ and $h\in  \opcop{H}$. Then$$(ab)\lhd h=h\cdot(ba)= (h_{(1)} \cdot b) (h_{(2)}\cdot a)= (a\lhd {h_{[1]}})(b\lhd {h_{[2]}}).$$Moreover, if $a\in \op{A}$ and $k,h\in \opcop{H}$ then  $$(a\lhd  k)\lhd h= h\cdot(k\cdot a)= (h_{(1)}\cdot 1_A )(h_{(2)}k\cdot a)=(a\lhd k{h_{[1]}})(1_A\lhd{h_{[2]}}).$$ Similarly, one can verify that  $(a\lhd  k)\lhd h= h\cdot(k\cdot a)= (1_A\lhd{h_{[1]}})(a\lhd k{h_{[2]}})$.
	\begin{remark}
		
		According to  \cite{ABV}, if $A$
		and $B$ are partial left $H$-module algebras, one can define a morphism of partial $H$-module algebras as an algebra morphism $f:A \to B$ such that $f(h \cdot a) =h \cdot f(a)$, for all $a\in A$ and $h\in H$. The category of all symmetric partial left $H$-module algebras and the morphisms of partial $H$-module algebras between them is denoted as LParAct$H.$ Analogously, one defines the category RParAct$H$ of symmetric  right partial $H$-module algebras.  By the previous example, the categories LParAct$H$  and  RParAct${\op{H}}$ are isomorphic.  
	\end{remark}
	We proceed with the next.
	\begin{definition}
		{\rm  
			Let $H$ be a Hopf algebra and $B$ be a unital algebra. A  {\it partial
				representation of $H$ in $B$} is a linear map $\pi:H\to B$ such that
			\begin{enumerate}
				\item[\footnotesize{(PR1)}]$\,$ $\pi(1_H) =1_B$, \vspace{.1cm}
				\item[\footnotesize{(PR2)}]$\,$ $ \pi(h)\pi(k_{(1)})\pi(S(k_{(2)}))=\pi(hk_{(1)})\pi(S(k_{(2)}))$, \vspace{.1cm}
				\item[\footnotesize{(PR3)}]$\,$ $\pi(h_{(1)})\pi(S(h_{(2)}))\pi(k)=\pi(h_{(1)})\pi(S(k_{(2)})k)$,
			\end{enumerate}
			for any $h,k \in H.$}
		
	\end{definition}  
	
	\begin{remark}
		Let $H$ be a Hopf algebra, $B$ be a unital algebra, and $\pi:H\to B$ be  a partial representation. It follows by  Lemma 2.11 of \cite{ABCQV}  that the following assertions also hold: for all $h,k \in H$,
		\begin{enumerate}
			\item[\footnotesize{(PR4)}]$\,$$ \pi(h)\pi(S(k_{(1)}))\pi(k_{(2)})=\pi(hS(k_{(1)}))\pi(k_{(2)})$,\vspace{.1cm}
			\item[\footnotesize{(PR5)}]$\,$$ \pi(S(h_{(1)}))\pi(h_{(2)})\pi(k)=\pi(S(h_{(1)}))\pi(h_{2}k)$.
		\end{enumerate}
	\end{remark}
	
	\begin{example}\label{rrep1}{\rm
			Let $H$ be a Hopf algebra, $B$ be  a unital algebra and $\pi: H\to B$  be a partial representation. Consider  the linear map  $\overline{\pi}: \opcop H\to \op{B}$ given by
			\[\overline{\pi}(h)=\pi(h),\quad h\in \opcop{H}.\] 
			We claim that $\overline{\pi}$ is a partial representation. Notice that (PR1) is clear. To show  (PR2), consider $h,k\in\opcop H$. Then
			\begin{align*}
				\overline\pi(h)\cdot_{{\rm op}}\overline\pi(k_{[1]})\cdot_{{\rm op}}\overline\pi(S(k_{[2]}))&=\,\,\pi(S(k_{(1)}))\pi(k_{(2)})\pi(h)
				\\&\overset{\mathclap{\scriptscriptstyle{\rm (PR5)}}}{=}\,\,\pi(S(k_{(1)}))\pi(k_{(2)}h)
				\\&= \,\,\overline\pi(h\cdot_{{\rm op}}k_{[1]})\cdot_{{\rm op}}\overline\pi(S(k_{[2]})).
			\end{align*}
			Similarly, using that $\pi$ satisfies (PR4), we obtain that $\overline{\pi}$ satisfies (PR3).}
	\end{example}

	\begin{example}\label{rrep}{\rm
			Let $H$ be a Hopf algebra and $\cdot: H\otimes A\to A$ be a left partial action of $H$ on an algebra $A$. By Example 3.5 in \cite{ABV}, the linear map \begin{align} \label{rep-induced}
				\pi: H\to {\rm End}(A),\quad  \pi(h)(a)=h\cdot a,\quad  h\in H,\,\, a\in A,
			\end{align}
			is a partial representation of $H$ in  ${\rm End}(A)$. It follows from the previous example that the linear map $\overline{\pi}: \opcop H\to \op{{\rm End}(A)}$ given by
			\[\overline{\pi}(h)=\pi(h),\quad h\in \opcop{H},\] 
			is a partial representation. }
	\end{example}
	
	\subsection{The partial smash product}
	
	Let $\cdot: H\otimes A\to A$ be a partial action of a Hopf algebra $H$ over an algebra $A$. On the vector space $A\otimes  H$, one can define the following associative product $$(a\otimes h)(b\otimes k)=a(h_{(1)}\cdot b)\otimes (h_{(2)}k),\quad a,b\in A,\,\,h,k\in H.$$  
	Following \cite{CJ}, the {\it (left) partial smash product} is the vector subspace $$\underline{A\# H}:=(A\otimes  H)(1_A\otimes 1_H)$$ of $A\otimes H$ which is  generated by elements of the form 
	\[a\#h=a(h_{(1)}\cdot 1_A)\otimes h_{(2)}, \quad a\in A,\,h\in H.\]
	Notice that, by construction, $\underline{A\# H}$ is a unital algebra and $1_{\underline{A\# H}}=1_A\#1_H$. \smallbreak

	We need the following notion to recall the universal property of the partial smash product proved in \cite{ABV}.
	
	\begin{definition}{\rm
			Let $A$ and $B$ be algebras, $H$ be a Hopf algebra and $\cdot:H\otimes A\to A$ be a partial action of $H$ on $A$. A {\it left covariant pair} associated to these data is a pair of maps $(\psi, \pi),$ where $\psi:A\to B$ is an algebra morphism and $\pi:H\to B$ is a partial representation, that satisfies: for all $h\in H$ and $a\in A$, 
			\begin{enumerate}
				\item[\footnotesize{(CP1)}]$\,$ $\psi(h\cdot a) =\pi(h_{(1)})\psi(a)\pi(S(h_{(2)})$,\vspace{.1cm}
				\item[\footnotesize{(CP2)}]$\,$ $\psi(a)\pi(S(h_{(1)}))\pi(h_{(2)})=\pi(S(h_{(1)}))\pi(h_{(2)})\psi(a).$
		\end{enumerate} }
	\end{definition}
	
	Consider the following linear maps
	\begin{align*}
		&\phi_0:A\to \underline{A\# H},\quad a\mapsto a\#1_H,& &\pi_0:H\to \underline{A\# H},\quad h\mapsto 1_A\#h.&
	\end{align*}
	It was proved in \cite{ABV} that $\phi_0$ is an algebra monomorphism and $\pi_0$ is a partial representation.
	The next result is Theorem 3.9 of \cite{ABV}.
	\begin{theorem}\label{luniv}
		Let $A$ and $B$ be unital algebras and $\cdot:H\otimes A\to A$ be a  left
		partial action of the Hopf algebra $H$ on $A$. Suppose that $(\psi,\pi)$ is a covariant pair associated to these data.
		Then the map
		\begin{equation*}\label{lphi}
			\Phi:\underline{A\# H}\to B,\,\,\, a\#h\to \psi(a)\pi(h),\,\,\, a\in A,\,\,\, h\in H,
		\end{equation*}  is the unique algebra morphism such that $\psi=\Phi\circ\phi_0$ and $\pi=\Phi\circ\pi_0.$
	\end{theorem}
	%\begin{example}\label{lsmash} Let $\cdot:H\otimes A\to A$ be a left
	%partial action of the Hopf algebra $H$ on an algebra $A$
	%$L$ be a left $A$-module and a $H$-partial module. It follows from 
	%Example 3.5 of \cite{ABV}  and from the proof of Theorem 3.10 in \cite{ABV} that the map 
	%$$\pi:H\to {\rm End}(L), \quad \pi(h)(l)=h\cdot l,\quad h\in H,\,\,l\in L $$ is a partial representation and $\psi:A\to {\rm End}(L), \psi(a)(l)=a l,\,\, a\in A,\,\, l\in L$ is an algebra morphism. Moreover, if $(\psi, \pi)$ satisfy, (CP1). Thus, by Theorem \ref{luniv}, $L$ is a left $\underline{A\# H}$-module with action $(a\#h)l=a(h\cdot l)$, for all $a\in A,\,\,h\in H,\,\, l\in L$.
	%\end{example}
	
	Consider now a right partial action $\cdot:A\otimes H\to A$ of a Hopf algebra $H$ on an algebra $A$. As in the left side case, one can define the right partial smash product. Precisely, one endows the vector space $H\otimes  A$ with  an associative product $$(k\otimes b)(h\otimes a)=kh_{(1)}\otimes (b\cdot h_{(2)})a,\quad a,b\in A,\,\,h,k\in H.$$ 
	\begin{definition}\label{right-smash}
		{\rm Let $\cdot:A\otimes H\to A$ be a  right partial action of a Hopf algebra $H$ on an algebra $A$.
			The {\it (right) partial smash product} is the vector subspace $$\underline{H\# A}:=(1_H\otimes 1_A)(H\otimes  A)$$ of $H\otimes A$ which is generated by typical elements of the form 
			\begin{equation}\label{elsm}
				h\#a=h_{(1)}\otimes (1_A\cdot h_{(2)})a, \,\,\, a\in A,\,h\in H.
			\end{equation}
			The right partial smash product $\underline{H\# A}$ is an associative unital algebra and its unity element is $1_{\underline{H\# A}}=1_H\#1_A$.
		}
	\end{definition}

	\begin{example} Suppose that  $ \alpha =\left(  A_g, \alpha_g  \right)_{g\in \tG}$ is a unital partial action of a group $\tG$ on an algebra $A$. As we saw in Example \ref{ex-pa-group},  we have a left partial action of $\Bbbk \tG$ on $A$ given by $g\cdot a=\alpha_g(a1_{g\m})$, for all $a\in A$ and $g\in \tG$. We recall from \cite{DE} that the {\it partial crossed product by $\alpha$} is the associative algebra $A\rtimes_{\alpha} \tG$ whose elements are finite formal sums in the form $\sum_{g\in \tG}a_g\delta_g$, where $a_g\in A_g$, and the multiplication is given by 
		\[a_g\delta_g\cdot a_h\delta_h=a_g\alpha_g(a_h1_{g\m})\delta_{gh}, \,\,\,g,h\in \tG, \,\,\,a_g\in A_g,\,\,a_h\in A_h.\]
		Using \eqref{elsm} we get  that  the map %$\eta:A\# \Bbbk \tG\to A\rtimes_{\alpha} \tG$ defined by
		\begin{equation}\label{lsm}
			\eta:\underline{A\# \Bbbk \tG}\to A\rtimes_{\alpha} \tG, \,\,a\# g\mapsto a1_{g}\delta_{g},\quad g\in \tG,\,\,a\in A,\end{equation} is an algebra isomorphism with inverse $b\delta_g\mapsto b\# g, \,\,b\in A_g, \, g\in \tG.$ Moreover,
		it follows by Example \ref{ltor}, that the partial action $\alpha$ of $\tG$ on $A$ induces a right partial action of $(\Bbbk G)^{\rm op}=(\Bbbk G)^{\rm opcop}$ on $\op{A}$ via $a\cdot g=\alpha_g(a1_{g\m})$, for all $a\in A$ and $g\in G$. Also, by \eqref{elsm} it is not difficult  to verify that 
		\begin{equation}\label{rsm} \lambda:\underline{(\Bbbk \tG)^{\rm op}\# A}\to ( \op A\ltimes_{\alpha} \tG)^{\rm op},\,\,\,    g\# a\mapsto a1_{g}\delta_{g},\quad g\in \tG,\,\,a\in A,\end{equation} 
		is an algebra isomorphism.
	\end{example}
	\begin{definition}\label{cov-case-op}
		{\rm Let $A$ and $B$ be algebras, $H$ be a Hopf algebra and $\cdot:A\otimes \op{H}\to A$ be a right partial action of $\op{H}$ on $A$. A {\it opposite covariant pair} associated to these data is a pair of maps $(\varphi, \gamma),$ where $\varphi:A\to {B}$ is an algebra morphism and $\gamma:\opcop{H}\to {B}$ is a partial representation, that satisfies: for all $h\in \opcop{H}$ and $a\in A$, 
			\begin{enumerate}
				\item [\footnotesize{(OCP1)}]$\,$  $\varphi(a\cdot h) =\gamma(S\m(h_{[2]}))\varphi(a)\gamma(h_{[1]}))$,\vspace{.1cm}
				\item [\footnotesize{(OCP2)}]$\,$ $\varphi(a)\gamma(\copr{h}{[2]})\gamma(S\m(h_{[1]}))=\gamma(h_{[2]})\gamma(S\m(h_{[1]}))\varphi(a)$.
		\end{enumerate} }
	\end{definition}

	\begin{lemma} Let $H$ be a Hopf algebra and $A$ be a right partial $\op H$-module. Then the following assertions hold:
		\begin{enumerate}[\rm (i)]
			\item the linear map $\varphi_0:A\to \underline{\op H\# A},\,\, a\mapsto 1_H\#a$, is an algebra monomorphism,\vspace{.1cm}
			\item the linear map $\gamma_0:\opcop H\to \underline{\op H\# A},\,\, h\mapsto h\#1_A$, is a partial representation,\vspace{.1cm}
			\item $(\varphi_0, \gamma_0)$ is an opposite covariant pair.
		\end{enumerate}    
	\end{lemma}
	
	\begin{proof}
		It is straightforward to verify that (i) is true. To prove (ii), we observe that $\gamma_0(1_H)=1_H\#1_A=1_{\underline{\op H\# A}}$. Also, given $h,k\in \opcop{H}$, we have that
		\begin{align*}
			\gamma_0(h)\gamma_0(k_{[1]})\gamma_0(S(k_{[2]}))&=(h\#1_A)(\copr{k}{(2)}  \# 1_A)(S(\copr{k}{(1)}) \# 1_A )
			\\&=(\copr{k}{(2)}h  \# 1_A\cdot\copr{k}{(3)})(S(\copr{k}{(1)}) \# 1_A )
			\\&=S(\copr{k}{(2)})\copr{k}{(3)}h  \# (1_A\cdot\copr{k}{(4)})\cdot S(\copr{k}{(1)})
			\\&=h  \# (1_A\cdot\copr{k}{(2)})\cdot S(\copr{k}{(1)})
			\\&=h  \# (1_A\cdot S(\copr{k}{(2)})\copr{k}{(3)})(1_A\cdot S(\copr{k}{(1)})) \qquad (\text{by } {\rm{({RPA}3)}})
			\\&=h  \# 1_A\cdot S(k).
		\end{align*}
		On the other hand,
		\begin{align*}
			\gamma_0(h\cdot_{op} k_{[1]})\gamma_0(S(k_{[2]}))&=(\copr{k}{(2)}h\#1_A)(S(k_{(1)})\#1_A)
			\\&=S(k_{(2)}) \copr{k}{(3)}h\#1_A\cdot S(k_{(1)}) 
			\\&=h\#1_A\cdot S(k), 
		\end{align*}   
		and consequently (PR2) holds. Finally, notice that 
		\begin{align*}
			\gamma_0(\copr{h}{[1]})\gamma_0(S(\copr{h}{[2]}))\gamma_0(k)&= (\copr{h}{(2)}  \# 1_A)(S(\copr{h}{(1)}) \# 1_A )(k\#1_A)   \\&=  (S (\copr{h}{(2)})\copr{h}{(3)}\# 1_A\cdot S (\copr{h}{(2)}))(k\#1_A) 
			\\&=  (1_H\# 1_A\cdot S (h))(k\#1_A)  
			\\&=  \copr{k}{(1)}\#( 1_A\cdot S (h))\cdot \copr{k}{(2)}, 
		\end{align*}
		and
		\begin{align*}
			\gamma_0(\copr{h}{[1]})\gamma_0(S(\copr{h}{[2]}) \cdot_{op} k)&= ( \copr{h}{(2)}  \# 1_A)(kS(\copr{h}{(1)} ) \# 1_A)
			\\&=k_{(1)}S(\copr{h}{(2)} )\copr{h}{(3)}\# 1_A\cdot k_{(2)}S(\copr{h}{(1)} )
			\\&=k_{(1)}\# 1_A\cdot k_{(2)}S(h) 
			\\&=k_{(1)}\# 1_A\cdot (S(h)\cdot_{\rm op}k_{(2)})
			\\&=k_{(1)}\# (1_A\cdot k_{(2)})(1_A\cdot (S(h)\cdot_{\rm op} k_{(3)}) \,\qquad \text{(by \eqref{elsm})}
			\\&=k_{(1)}\# (1_A\cdot S (h))\cdot \copr{k}{(2)},\qquad \quad\qquad\qquad (\text{by } {\rm{({RPA}3)}})
		\end{align*}
		which implies that (PR3) holds, and hence (ii) is proved. To prove (iii) we note that
		\begin{align*}
			\gamma_0(S\m(h_{[2]}))\varphi_0(a)\gamma_0(h_{[1]}))&=(S\m(h_{(1)})\#1_A)(1_H\# a) (h_{(2)}\#1_A)
			\\ &=(S\m(h_{(1)})\# a)(h_{(2)}\#1_A)   
			\\&=h_{(2)}S\m(h_{(1)})\# a\cdot h_{(3)}
			\\&=1_H\# a\cdot h
			\\&=\varphi_0(a\cdot h).\end{align*} 
		Also,
		\begin{align*}\varphi_0(a)\gamma_0(\copr{h}{[2]})\gamma_0(S\m(h_{[1]}))\,\,&=\,\,(1_H\# a)(h_{(1)}\#1_A)(S\m(h_{(2)})\#1_A)
			\\ &=\,\,(h_{(1)}\# a\cdot h_{(2)})(S\m(h_{(3)})\#1_A)
			\\ &= \,\,S\m(h_{(4)})h_{(1)}\# (a\cdot h_{(2)})\cdot S\m(h_{(3)})
			\\ &\overset{\mathclap{\scriptscriptstyle{\rm (RPA3)}}}{=}\,\, S\m(h_{(5)})h_{(1)}\# (1_A\cdot S\m(h_{(4)}))(a\cdot (h_ {(2)}\cdot_{\rm op}S\m(h_{(3)})) )
			\\ &=\,\, S\m(h_{(5)})h_{(1)}\# (1_A\cdot S\m(h_{(4)}))(a\cdot S\m(h_{(3)})h_ {(2)} )
			\\ &=\,\, S\m(h_{(3)})h_{(1)}\# (1_A\cdot S\m(h_{(2)}))a,
		\end{align*} 
		and 
		\begin{align*}\gamma_0(h_{[2]})\gamma_0(S\m(h_{[1]}))\varphi_0(a)&=(h_{(1)}\#1_A)(S\m(h_{(2)})\#1_A)(1_H\# a)
			\\ &=(h_{(1)}\#1_A)(S\m(h_{(2)})\#a)
			\\ &=S\m(h_{(3)})h_{(1)}\#(1_A\cdot S\m(h_{(2)}))a,
		\end{align*}
		then $\gamma_0(h_{[2]})\gamma_0(S\m(h_{[1]}))\varphi_0(a)=\varphi_0(a)\gamma_0(\copr{h}{[2]})\gamma_0(S\m(h_{[1]})),$ which finishes the proof.
	\end{proof}
	
	\vspace{.2cm}
	
	We now establish a universal property for opposite covariant pairs, analogous to the version for covariant pairs presented in Theorem 2.25.

	\begin{theorem}\label{runiv}
		Let $A$ and $B$ be unital algebras and $\cdot: A\otimes \op H\to A$ be a  right
		partial action of the Hopf algebra $H$ on $A$. Suppose that $(\varphi,\gamma)$ is an opposite covariant pair associated with these data.
		Then there exists a unique algebra morphism $\Gamma:\underline{\op H\#A} \to B$ such that $\varphi=\Gamma\circ\varphi_0$ and $\gamma=\Gamma\circ\gamma_0$.
	\end{theorem}
	\begin{proof}
		
		Consider the linear map  
		\begin{equation}\label{gammar}
			\Gamma:\underline{\op H\#A} \to B,\qquad \Gamma(h\#a)=\gamma(h)\varphi(a) , \quad h\in H,\,\,a\in A.   \end{equation}
		It is straightforward to see that $\varphi=\Gamma\circ\varphi_0$ and $\gamma=\Gamma\circ\gamma_0$. Given $k,h\in H$ and $a,b\in A$, we have 
		\begin{align*}
			\Gamma((k\#b)(h\# a))&=\Gamma(k\cdot_{\rm op} h_{(1)}\# (b\cdot h_{(2)})a)
			\\&= \gamma(k\cdot_{\rm op} h_{(1)})\varphi((b\cdot h_{(2)})a)
			\\&= \gamma(k\cdot_{\rm op} h_{[2]})\varphi(b\cdot h_{[1]})\varphi(a)
			\\&=\gamma(k\cdot_{\rm op} h_{[3]})\gamma(S\m(h_{[2]})) \varphi(b)\gamma(h_{[1]})\varphi(a)\qquad\quad\qquad\, (\text{by \footnotesize{(OCP1)}})
			\\&=\gamma(k)\gamma( h_{[3]})\gamma(S\m(h_{[2]})) \varphi(b)\gamma(h_{[1]})\varphi(a)\qquad\quad\qquad\,\, (\text{by \footnotesize{(PR4)}})
			\\&=\gamma(k)\varphi(b) \gamma(h_{[3]})\gamma(S\m(h_{[2]}))\gamma(h_{[1]})\varphi(a) \qquad\quad\qquad\,\,  (\text{by \footnotesize{(OCP2)}})
			\\&=\gamma(k)\varphi(b) \gamma(h_{[3]})\gamma(S\m(h_{[2]}))\gamma(S(S\m(h_{[1]})))\varphi(a)
			\\&=\gamma(k)\varphi(b) \gamma(h_{[3]})\gamma(S\m(h)_{[1]})\gamma(S(S\m(h)_{[2]}))\varphi(a)
			\\&=\gamma(k)\varphi(b) \gamma(h_{[3]}\cdot_{\rm op}S\m(h_{[2]}))\gamma(h_{[1]})\varphi(a)\qquad\quad\qquad\,\,  (\text{by \footnotesize{(PR2)}})
			\\&=\gamma(k)\varphi(b)\varphi(a)\gamma(h)\varphi(a)
			\\&=\Gamma(k\#b)\circ \Gamma(h\#a).\end{align*} Since $\Gamma(1_H\#1_A)=1_B$ it follows that $\Gamma$ is a morphism of algebras.
		The proof of the uniqueness of $\Gamma$ is similar to the proof given in \cite[Theorem 3.9]{ABV}. 
	\end{proof}

	\section{The main result}
	
	Throughout this section, $H$ denotes a Hopf algebra with inversible antipode $S$ and $R=(\mdl {M}{i}{j})_{i,j\in \bI_n}$ denotes a generalized matrix algebra such that $\mdl{M}{i}{i}=R_i$ is a unital algebra. The identity element of $R_i$ is denoted by $1_i$. This section is dedicated to prove the main result of this work, which provides necessary and sufficient conditions for the existence of a partial action of $H$ on $R$.
	
	\smallbreak
	
	Given $i,j\in \bI_n$, we denote by $\iota_{ij}:\,\,\!\!\mdl{M}{i}{j}\to R$ the natural inclusion of $\mdl{M}{i}{j}$ in $R$, that is, $\iota_{ij}(m)$ is the matrix whose the entries are all zero except the $(i,j)$-entry which is equal to $m$, for all $m\in\, \!\mdl{M}{i}{j}$. By Remark \ref{bimodule-structure}, $R$ is an $(R_i,R_j)$-bimodule and it is easy to see that $\iota_{ij}$ (resp. $\iota_{ii}$) is a monomorphism of $(R_i,R_j)$-bimodules (resp. of algebras). From now on, we will denote by $\mdl{\tilde{M}}{i}{j}=\iota_{ij}(\mdl{M}{i}{j})$ the $(R_i,R_j)$-subbimodule of $R$ which is isomorphic to $\mdl{M}{i}{j}$.
	
	\begin{definition}\label{rem-restric}
		Let  $\rhd:H\otimes R\to R$ be a left partial action of $H$ on $R$ and $i,j\in \bI_n$. The subbimodule $\mdl{\tilde{M}}{i}{j}$ of $R$ is said {\it $H$-invariant} if $h\rhd m\in\,\! \mdl{\tilde{M}}{i}{j}$, for all $h\in H$ and $m\in\,\! \mdl{\tilde{M}}{i}{j}$.
	\end{definition}

	\begin{lemma}\label{lem-aux-main} 
		Let $\rhd: H\otimes R\to R$ be a left partial action of $H$ on $R$ such that $\mdl{\Tilde{M}}{i}{j}$ is invariant, for all $i,j\in \bI_n$. Then:
		\begin{enumerate} [\rm (i)]
			\item   the linear map $ \rightharpoonup_i :H\otimes R_i \rightarrow R_i$ given by
			\[
			h\rightharpoonup_i r=\iota^{-1}_{ii}\big( h\triangleright \iota_{ii}(r)\big), \quad h\in H,\,\,r\in R_i,
			\]
			is a partial action, \vspace{.1cm}
			\item   the linear map $\pi_{ij}:H\to \End{\mdl{M}{i}{j}}$ given by
			\[\pi_{ij}(h)(m)=\iota\m_{ij}\big(h\rhd \iota_{ij}(m)\big),\quad h\in H,\,\,\,m\in\,\!\mdl{M}{i}{j}, \]
			is a partial representation, \vspace{.1cm}
			\item the linear map $\gamma_{ij}:\opcop{H}\to \End{\mdl{M}{i}{j}}^{\rm op}$ given by
			\[\gamma_{ij}(h)(m)=\iota\m_{ij}\big(h\rhd \iota_{ij}(m)\big),\quad h\in H,\,\,\,m\in\,\!\mdl{M}{i}{j}, \]
			is a partial representation.
		\end{enumerate}
	\end{lemma}

	\begin{proof} (i) 
		Let $r\in R_i$. Then
		\[
		1_H \rightharpoonup_i r =\iota^{-1}_{ii}\big( 1_H \triangleright \iota_{ii}(r)\big) =\iota^{-1}_{ii}\big(  \iota_{ii}(r)\big) =r,
		\]
		which implies (LPA1).
		For (LPA2), consider $r,s \in R_i$ and $h\in H$. Thus
		\begin{align*}
			h \rightharpoonup_i (rs) & = 
			\iota^{-1}_{ii}\big( h \triangleright \iota_{ii}(rs)\big) \\
			& = \iota^{-1}_{ii}\big( h \triangleright (\iota_{ii}(r) \iota_{ii}(s))\big) \\
			& =  \iota^{-1}_{ii}\big( \left(h_{(1)} \triangleright \iota_{ii}(r)\right) \left(h_{(2)} \triangleright\iota_{ii}(s)\right)\big) \\
			& = \iota^{-1}_{ii} \left(h_{(1)} \triangleright \iota_{ii}(r)\right) \iota^{-1}_{ii}\left(h_{(2)} \triangleright\iota_{ii}(s)\right) \\
			& = \left( h_{(1)} \rightharpoonup_i r \right) \left( h_{(2)} \rightharpoonup_i r \right) .
		\end{align*}
		Finally, in order to verify (LPA3), consider $i\in \bI_n$, $h\in H$ and $r\in R_i$. Then
		\begin{eqnarray*}
			(h\triangleright 1_R)\iota_{ii}(r) = (h\triangleright \iota_{ii}(1_i))\iota_{ii} (r) ,\qquad\iota_{ii}(r)(h\triangleright 1_R) = \iota_{ii}(r)(h\triangleright \iota_{ii}(1_i)) .
		\end{eqnarray*}
		Then, for every $h,k\in H$ and $r\in R_i$
		\begin{align*}
			h\rightharpoonup_i (k\rightharpoonup_i r) & =\,\, \iota^{-1}_{ii}\big( h \triangleright \big( k\triangleright \iota_{ii}(r)\big) \big) \\
			& = \,\, \iota^{-1}_{ii}\big( \left( h_{(1)}  \triangleright 1_R \right) \left( h_{(2)}k\triangleright \iota_{ii}(r)\right) \big) \\
			& \stackrel{(\ast)}{=}  \iota^{-1}_{ii}\big( \left( h_{(1)}  \triangleright \iota_{ii}(1_i) \right) \left( h_{(2)}k\triangleright \iota_{ii}(r)\right) \big) \\
			& = \,\, \iota^{-1}_{ii} \left( h_{(1)}  \triangleright \iota_{ii}(1_i) \right) \iota^{-1}_{ii}\left( h_{(2)}k\triangleright \iota_{ii}(r)\right)  \\
			& =  \,\left( h_{(1)} \rightharpoonup_i 1_i \right)\left( h_{(2)}k \rightharpoonup_i r \right),
		\end{align*}
		where ($\ast$) follows from the $H$-invariance of $\tilde{R}_i$. Similarly, one verifies that
		\[
		h\rightharpoonup_i (k\rightharpoonup_i r)=\left( h_{(1)}k \rightharpoonup_i r \right)\left( h_{(2)} \rightharpoonup_i 1_i \right)
		\]
		(ii)
		For any $m\in {}_i M_j$, we have that
		\[\pi_{ij}(1_H)(m)=\iota\m_{ij}\big(1_H\rhd \iota_{ij}(m)\big)=\iota\m_{ij}\big(\iota_{ij}(m)\big)=m,\]
		and hence (PR1) holds. Also, given $h,k\in H$, we have that
		\begin{align*}
			\pi_{ij}(h)\pi_{ij}(k_{(1)})\pi_{ij}(S(k_{(2)}))(m)&=\iota\m_{ij}\big(h\rhd\big(k_{(1)}\rhd(S(k_{(2)})\rhd\iota_{ij}(m))\big)\big)\\
			&=\iota\m_{ij}\big(h\rhd\big((k_{(1)}\rhd 1_R)(k_{(2)}S(k_{(3)})\rhd\iota_{ij}(m))\big) \big)\, (\text{by \footnotesize{(LPA3)}})\\
			&=\iota\m_{ij}\big(h\rhd\big((k\rhd 1_R)\iota_{ij}(m)\big)\big),
		\end{align*}
		and, on the other hand,
		\begin{align*}
			\pi_{ij}(hk_{(1)})\pi_{ij}(S(k_{(2)}))(m)&=\iota\m_{ij}\big(hk_{(1)}\rhd(S(k_{(2)})\rhd\iota_{ij}(m))\big)\\
			&=\iota\m_{ij}\big((h_{(1)}k_{(1)}\rhd 1_R)(h_{(2)}k_{(2)}S(k_{(3)})\rhd\iota_{ij}(m))\big)\quad\,(\text{by \footnotesize{(LPA3)}})\\
			&=\iota\m_{ij}\big((h_{(1)}k\rhd 1_R)(h_{(2)}\rhd\iota_{ij}(m))\big)\\
			&=\iota\m_{ij}\big((h_{(1)}\rhd (k\rhd 1_R))(h_{(2)}\rhd\iota_{ij}(m))\big)\\
			&=\iota\m_{ij}\big(h\rhd \big( (k\rhd 1_R)\iota_{ij}(m)\big)\big).
		\end{align*}
		Thus, (PR2) is true. Similarly, we can verify that (PR3) also holds, and consequently (ii) is proved. Item (iii) follows from (ii) and Example \ref{rrep1}.    
	\end{proof}

	\begin{theorem}\label{teo-pa-gma}
		The following assertions are equivalent:
		\begin{enumerate} [\rm (i)]
			\item there exists a  left partial action $\rhd: H\otimes R\to R$ of $H$ on $R$ such that $\mdl{\Tilde{M}}{i}{j}$ is invariant, for all $i,j\in \bI_n$,\vspace{.1cm}
			\item  the following conditions hold:\vspace{.1cm}
			\begin{enumerate} [\rm (a)]
				\item for each $i\in \bI_n$, there is a left partial action $\rightharpoonup_{i}: H\otimes R_i\to R_i$ of $H$ on $R_i$, \vspace{.1cm}
				\item $\mdl{M}{i}{j}$ is a left $\underline{R_i\# H}$-module and a right $\underline{\op{H}\# R_j}$-module such that 
				\begin{align*}
					& (1_i\#h)\cdot m=m\cdot (h\# 1_j),& &\,\,(1_i\#h)\cdot r= h\rightharpoonup_{i} r,&\\[.2em]
					&\,\,\,(r_i\# 1_H)\cdot m=r_i\cdot m,& &m\cdot (1_H\#r_j)=m\cdot r_j,&
				\end{align*}
				for all $h\in H$, $m\in\,\! \mdl{M}{i}{j}$, $r,r_i\in R_i$, $r_j\in R_j$ and $i,j\in \bI_n$,\vspace{.2cm}
				\item $(1_i\#h)\cdot mn=\big((1_i\#h_{(1)})\cdot m\big)\big((1_j\#h_{(2)})\cdot n\big),\,\,\, \text{for all }\,h\in H,\,m\in\,\!\! \mdl{M}{i}{j},\,n\in\,\!\! \mdl{M}{j}{k}$.
			\end{enumerate}
		\end{enumerate}
	\end{theorem}  \vspace{.05cm}
	
	\begin{proof} 
		(i) $\Rightarrow$ (ii) Let $i\in\bI_n$. Since $\mdl{\Tilde{M}}{i}{i}$ is invariant by the partial action of $H$ on $R$, it follows from  Lemma \ref{lem-aux-main} (i) that the restriction $\rightharpoonup_{i}:H\otimes R_i\to R_i$ is a partial action of $H$ on $R_i$ and hence (a) is proved.
		To prove (b), we recall that $\mdl{M}{i}{j}$ an $(R_i,R_j)$-bimodule, for all $i,j\in \bI_n$. In particular, $\mdl{M}{i}{j}$ is a left $R_i$-module and consequently there is an algebra morphism $\psi_{ij}:R_i\to \End{\mdl{M}{i}{j}}$, given by
		\[
		\psi_{ij} (r)(m) =\iota^{-1}_{ij}\big( \iota_{ii}(r)\iota_{ij}(m) \big)=\iota^{-1}_{ij}\big( \iota_{ij}(r\cdot m)\big)=r\cdot m, \quad \text{for all } r\in R_i,\,\, m\in {}_i M_j .
		\]
		Moreover, by Lemma \ref{lem-aux-main} (ii), $\pi_{ij}:H\to \End{\mdl{M}{i}{j}}$ given by
		\[\pi_{ij}(h)(m)=\iota\m_{ij}\big(h\rhd \iota_{ij}(m)\big),\qquad h\in H,\,\,\,m\in\,\!\mdl{M}{i}{j}, \]
		is a partial representation. One needs to verify that $(\psi_{ij},\pi_{ij})$ is a covariant pair. In order to prove (CP1), take $r\in R_i$, $h\in H$ and $m\in {}_i M_j$. Then
		\begin{align*}
			\psi_{ij} (h\rightharpoonup_i r)(m) & =  \iota^{-1}_{ij}\left( \iota_{ii}(h\rightharpoonup_i r)\iota_{ij} (m) \right) \\
			& =  \iota^{-1}_{ij}\left( \iota_{ii} \iota^{-1}_{ii} (h\triangleright \iota_{ii}(r)) \iota_{ij} (m)\right) \\
			& =  \iota^{-1}_{ij}\left(  (h\triangleright \iota_{ii}(r)) \iota_{ij} (m)\right) \\
			& =  \iota^{-1}_{ij}\left(  (h_{(1)}\triangleright \iota_{ii}(r)) (h_{(2)}S(h_{(3)})\triangleright \iota_{ij} (m))\right) \\
			& =  \iota^{-1}_{ij}\left(  (h_{(1)}\triangleright \iota_{ii}(r)) (h_{(2)}\triangleright (S(h_{(3)})\triangleright \iota_{ij} (m)))\right) \\
			& =  \iota^{-1}_{ij}\left(  h_{(1)}\triangleright \left( \iota_{ii}(r)  (S(h_{(2)})\triangleright \iota_{ij} (m))\right)\right) \\
			& =  \iota^{-1}_{ij}\left(  h_{(1)}\triangleright \left( \iota_{ii}(r)  \iota_{ij}\iota^{-1}_{ij}(S(h_{(2)})\triangleright \iota_{ij} (m))\right)\right) \\
			& =  \iota^{-1}_{ij}\left(  h_{(1)}\triangleright \left( \iota_{ii}(r)  \iota_{ij}(\pi_{ij}(S(h_{(2)}))(m))\right)\right) \\
			& =  \iota^{-1}_{ij}\left(  h_{(1)}\triangleright \iota_{ij}\iota^{-1}_{ij}\left( \iota_{ii}(r)  \iota_{ij}(\pi_{ij}(S(h_{(2)}))(m))\right)\right) \\
			& =  \iota^{-1}_{ij}\left(  h_{(1)}\triangleright \iota_{ij}\left( \psi_{ij}(r)  (\pi_{ij}(S(h_{(2)}))(m))\right)\right) \\
			& =  \pi_{ij}
			(h_{(1)})\left( \psi_{ij}(r)  (\pi_{ij}(S(h_{(2)}))(m))\right) \\
			& =  \pi_{ij}
			(h_{(1)})\circ \psi_{ij}(r)  \circ \pi_{ij}(S(h_{(2)}))(m) .
		\end{align*}
		For (CP2), on the one hand we have
		\begin{eqnarray*}
			\psi_{ij} (r) \circ \pi_{ij} (S(h_{(1)})) \circ \pi_{ij} (h_{(2)}) (m) & = & \psi_{ij} (r) \circ \pi_{ij} (S(h_{(1)})) \left( \iota^{-1}_{ij} ( h_{(2)}\triangleright \iota_{ij} (m))\right) \\
			& = & \psi_{ij} (r) \left( \iota^{-1}_{ij} \left( S(h_{(1)})\triangleright   ( h_{(2)}\triangleright \iota_{ij} (m))\right) \right)\\
			& = & \psi_{ij} (r) \left( \iota^{-1}_{ij} \left(\left( S(h_{(2)})h_{(3)}\triangleright \iota_{ij} (m)\right)\left( S(h_{(1)}) \triangleright 1_R \right) \right)\right) \\
			& = & \psi_{ij} (r) \left( \iota^{-1}_{ij} \left(  \iota_{ij} (m)\left( S(h) \triangleright 1_R \right) \right)\right)\\
			& = & \iota^{-1}_{ij} \left(\iota_{ii} (r)   \iota_{ij} (m)\left( S(h) \triangleright 1_R \right) \right) .
		\end{eqnarray*}
		On the other hand, 
		\begin{eqnarray*}
			\pi_{ij} (S(h_{(1)})) \circ \pi_{ij} (h_{(2)}) \circ \psi_{ij} (r) (m) 
			& = & \pi_{ij} (S(h_{(1)})) \circ \pi_{ij} (h_{(2)}) \left( \iota^{-1}_{ij} (\iota_{ii}(r) \iota_{ij}(m)) \right) \\
			& = & \pi_{ij} (S(h_{(1)})) \left(  \iota^{-1}_{ij} \left( h_{(2)}\triangleright (\iota_{ii}(r) \iota_{ij}(m)) \right)\right) \\
			& = &  \iota^{-1}_{ij} \left( S(h_{(1)})\triangleright \left( h_{(2)}\triangleright (\iota_{ii}(r) \iota_{ij}(m)) \right)\right) \\
			& = &  \iota^{-1}_{ij} \left(\left( S(h_{(1)})h_{(2)}\triangleright (\iota_{ii}(r) \iota_{ij}(m)) \right) \left( S(h_{(1)})\triangleright 1_R \right)\right) \\
			& = &  \iota^{-1}_{ij} \left( \iota_{ii}(r) \iota_{ij}(m) \left( S(h)\triangleright 1_R \right)\right) .
		\end{eqnarray*}
		Therefore, $(\psi_{ij} , \pi_{ij})$ is a covariant pair. It follows from Theorem \ref{luniv} that
		\[\Phi_{ij}: R_i\#H\to \End{\mdl{M}{i}{j}},\quad \Phi_{ij}(r\#h)=\psi_{ij}(r)\pi_{ij}(h), \,\,r\in R_i,\,\,h\in H,\]
		is an algebra morphism. Therefore, $\mdl{M}{i}{j}$ is a left $R_i\#H$-module. \smallbreak
		
		Now, we will denote by $\cdot: \,\!\mdl{M}{i}{j}\otimes R_j\to \,\!\mdl{M}{i}{j}$ the right action of $R_j$ on $\mdl{M}{i}{j}$. Then, $\varphi_{ij}:R_j\to \End{\mdl{M}{i}{j}}^{\rm op}$ given by $\varphi_{ij}(r)(m)=m\cdot r$, $r\in R_j$ and $m\in\,\! \mdl{M}{i}{j}$, is an algebra morphism.  
		Also, by Lemma \ref{lem-aux-main} (ii), the map
		$\gamma_{ij}:\opcop{H}\to \End{\mdl{M}{i}{j}}^{\rm op}$  given by
		\[\gamma_{ij}(h)(m)=\iota\m_{ij}\big(h\rhd \iota_{ij}(m)\big),\qquad h\in H,\,\,\,m\in\,\!\mdl{M}{i}{j}, \]
		is a partial representation. We claim that $(\varphi_{ij},\gamma_{ij})$ is an opposite covariant pair. 
		Firstly, by Example \ref{ltor}, the linear map $\leftharpoonup_{j}:R_j\otimes  \op H\to R_j,$ given by  
		$$r \leftharpoonup_{j} h:=h\rightharpoonup_{j}r=\iota\m_{jj}(h\rhd \iota_{jj}(r)),\,\,\, \text{for all}\,\, r\in R_j,\,\, h\in \op H,$$ defines a right partial action of $\op H$ in $R_j$.
		For all $r\in R_j$ and $m\in\!\mdl{M}{i}{j}$, we will denote by  $u(m,r):=\gamma_{ij}(S\m(h_{[2]}))\cdot_{\rm op}\varphi_{i,j}(r)\cdot_{\rm op}\gamma_{ij}(h_{[1]})(m)$ the element of $\mdl{M}{i}{j}$  obtained using the opposite product of $\End{\mdl{M}{i}{j}}$. Then
		\begin{align*}
			u(m,r)&=\,\,\gamma_{ij}(h_{[1]})\varphi_{i,j}(r)\gamma_{ij}(S\m(h_{[2]}))(m)\\
			&=\,\,\iota\m_{ij}\Big(h_{(2)}\rhd \iota_{ij}\big(\big(\iota\m_{ij}  \big(S\m(h_{(1)})\rhd \iota_{ij}(m)\big)\cdot r\big) \big)\Big) \\
			&=\,\,\iota\m_{ij}\big(h_{(2)}\rhd \big(S\m(h_{(1)})\rhd \iota_{ij}(m)\big)\cdot r \big)\\
			&=\,\,\iota\m_{ij}\big(h_{(2)}\rhd \big(S\m(h_{(1)})\rhd \iota_{ij}(m)\big)\iota_{jj}(r) \big)\\
			&=\,\,\iota\m_{ij}\Big(\big(h_{(2)}\rhd \big( S\m(h_{(1)})\rhd \iota_{ij}(m)\big)(h_{(3)}\rhd \iota_{jj}(r) \big)\Big)\\ 
			&\overset{\mathclap{\eqref{prodc}}}{=}\,\,\iota\m_{ij}\big(\big(h_{(2)} S\m(h_{(1)})\rhd \iota_{ij}(m)\big)(h_{(3)}\rhd \iota_{jj}(r) \big)\big)\\
			&=\,\,\iota\m_{ij}\big(\iota_{ij}(m)\big(h\rhd \iota_{jj}(r) \big)\big)\\
			&\stackrel{(*)}=\,\,m\cdot \iota\m_{jj}(h\rhd \iota_{jj}(r) \big)\big)\\
			&=\,\, \varphi_{ij}( r\leftharpoonup_{j} h)(m),\end{align*}
		where $(*)$ holds because there exists a unique $r'\in R_j$ such that $h\rhd \iota_{jj}(r)=\iota_{jj}(r')$ and consequently
		\[\iota\m_{ij}\big(\iota_{ij}(m)\big(h\rhd \iota_{jj}(r) \big)\big)=\iota\m_{ij}\big(\iota_{ij}(m)\iota_{jj}(r')\big)=\iota\m_{ij}\big(\iota_{ij}(m\cdot r')\big)=m\cdot r'.\]
		Thus, {\footnotesize{(OCP1)}} is true. In order to prove {\footnotesize{(OCP2)}} , given $r\in R_j$ and $m\in\!\mdl{M}{i}{j}$ we will denote  $v(m,r)=\varphi_{ij}(r)\cdot_{\rm op}\gamma_{ij}(\copr{h}{[2]})\cdot_{\rm op}\gamma_{ij}(S\m(h_{[1]})(m)\in\!\mdl{M}{i}{j}$. Notice that
		\begin{align*}
			v(m,r) &=\,\,\gamma_{ij}(S\m(h_{[1]}))\gamma_{ij}(\copr{h}{[2]})\varphi_{ij}(r)(m)
			\\&=\,\,\iota\m_{ij}\big(S\m(h_{(2)})\rhd (h_{(1)}\rhd\iota_{ij}(m\cdot r))\big)
			\\&=\,\,\iota\m_{ij}\big(S\m(h_{(2)})\rhd (h_{(1)}\rhd\iota_{ij}(m)\iota_{jj}(r))\big)
			\\&=\,\,\iota\m_{ij}\Big(S\m(h_{(3)})\rhd \big((h_{(1)}\rhd\iota_{ij}(m))(h_{(2)}\rhd\iota_{jj}(r))\big)\Big)
			\\&=\,\,\iota\m_{ij}\Big(\big(S\m(h_{(4)})\rhd (h_{(1)}\rhd\iota_{ij}(m))\big)\big(S\m(h_{(3)})\rhd (h_{(2)}\rhd\iota_{jj}(r))\big)\Big)
			\\&\overset{\mathclap{\eqref{prodc-1}}}{=}\,\,\iota\m_{ij}\Big(\big(S\m(h_{(4)})\rhd (h_{(1)}\rhd\iota_{ij}(m))\big)\big(h_{(2)}S\m(h_{(3)})\rhd \iota_{jj}(r)\big)\Big)
			\\&=\,\,\iota\m_{ij}\Big(\big(S\m(h_{(2)})\rhd (h_{(1)}\rhd\iota_{ij}(m))\big)\iota_{jj}(r)\Big)
			\\&=\,\,\iota\m_{ij}\Big(\big(S\m(h_{[1]})\rhd (h_{[2]}\rhd\iota_{ij}(m))\big)\iota_{jj}(r)\Big)
			\\&=\,\,\varphi_{ij}(r) \gamma_{ij}(S\m(h_{[1]}) \gamma_{ij}(\copr{h}{[2]})(m)
			\\&=\,\,\gamma_{ij}(\copr{h}{[2]})\cdot_{\rm op} \gamma_{ij}(S\m(h_{[1]}))\cdot_{\rm op} \varphi_{ij}(r)(m).\end{align*}
		Hence, by Theorem  \ref{runiv}, the map $\Gamma_{ij}:\underline{\op H\#R_j} \to \End{\mdl{M}{i}{j}}^{\rm op}$ given by  $$ \Gamma_{ij}(h\#r)=\gamma_{ij}(h)\cdot_{\rm op}\varphi_{ij}(r) , \quad h\in H,\,\,r\in R_j,$$ 
		is an algebra morphism. Thus, $\mdl{M}{i}{j}$ is a right $\underline{\op H\#R_j}$-module.  Now for  $h\in H$, $m\in\!\mdl{M}{i}{j}$ and $r\in R_i$,  we have
		\begin{align*}
			(1_i\#h)\cdot m&=\Phi_{ij}(1_i\#h)(m)=(\psi_{ij}(1_i)\pi_{ij}(h))(m)\\
			&=\iota\m_{ij}\big(h\rhd \iota_{ij}(m)\big)=\gamma_{ij}(h)(m)\\
			&=(\gamma_{ij}(h)\cdot_{\rm op}\varphi_{ij}(1_j))(m)=\Gamma_{ij}(h\#1_j)(m)\\
			&=m\cdot (h\#1_j),
		\end{align*}
		and 
		$$(1_i\#h)\cdot r =\iota\m_{ii}\big(h\rhd \iota_{ii}(r)\big)=h\rightharpoonup_{i} r.$$
		The others two identities of (b) follow directly from the definitions of $\Phi_{ij}$ and $\Gamma_{ij}$.
		Finally,
		\begin{align*}
			(1_i\#h)\cdot mn&=\iota\m_{ik}\big(h\rhd \iota_{ik}(mn)\big)
			\\&=\iota\m_{ik}\big(h\rhd \iota_{ij}(m)\iota_{jk}(n)\big)
			\\&=\iota\m_{ik}\big((h_{(1)}\rhd\iota_{ij}(m))(h_{(2)}\rhd \iota_{jk}(n))\big)
			\\&=\iota\m_{ij}(h_{(1)}\rhd\iota_{ij}(m))\iota\m_{jk}(h_{(2)}\rhd \iota_{jk}(n))
			\\&=  \big((1_i\#h_{(1)})\cdot m\big)\big((1_j\#h_{(2)})\cdot n\big),        \end{align*}
		for all $h\in H$, $m\in\,\! \mdl{M}{i}{j}$ and $n\in\,\! \mdl{M}{j}{k}$. Hence (c) is true and the implication (i) $\Rightarrow$ (ii) is proved.
		\vspace{.2cm}
		
		\noindent (ii) $\Rightarrow$ (i) Let $h\in H$ and $r=(\mdl{m}{i}{j})_{i,j\in \bI_n}\in R$. Using that $\mdl{M}{i}{j}$ is a left $\underline{R_i\# H}$-module, we will define the linear map $\rhd:H\otimes R\to R$ by
		\[h\rhd r=\big( (1_i\#h)\cdot\,\! \mdl{ m}{i}{j}\big)_{i,j\in \bI_n}.\]
		It is clear that {\footnotesize{(LPA1)}} is true. Also, the condition {\footnotesize{(LPA2)}} follows directly from (c). Now, consider $h,k\in H$ and $r=\big(\mdl{m}{i}{j}\big)_{i,j\in \bI_n}\in\,\!\mdl{M}{i}{j}$. Then
		\begin{align*}
			h\rhd\big(k\rhd r\big)&=h\rhd\Big( \big((1_i\# k)\cdot\,\! \mdl{m}{i}{j}\big)_{i,j\in \bI_n}  \Big)\\
			&=\Big( (1_i\# h)\cdot\big((1_i\# k)\cdot\,\! \mdl{m}{i}{j}\big)\Big)_{i,j\in \bI_n} \\
			&=\Big( (1_i\# h)(1_i\# k)\cdot\,\! \mdl{m}{i}{j}\Big)_{i,j\in \bI_n} \\
			&=\Big( \big( h_{(1)}\rightharpoonup_{i} 1_i\# h_{(2)}k\big)  \cdot\,\! \mdl{m}{i}{j}\Big)_{i,j\in \bI_n} ,
		\end{align*}
		\begin{align*}
			( h_{(1)}\rhd 1_R)\big(h_{(2)}k\rhd r\big)& = \operatorname{diag}((1_1\# h_{(1)})\cdot 1_1,\ldots,(1_n\# h_{(1)})\cdot 1_n)(h_{(2)}k\rhd r)\\
			&\overset{\mathclap{\rm(b)}}{=}\operatorname{diag}(h_{(1)}\rightharpoonup_{1} 1_1,\ldots,h_{(1)}\rightharpoonup_{n} 1_n)(h_{(2)}k\rhd r)\\
			&= \Big(\big((h_{(1)}\rightharpoonup_{i} 1_i)(1_i\#h_{(2)}k)\big)\cdot\,\! \mdl{m}{i}{j}\Big)_{i,j\in \bI_n}\\
			&=\Big( \big( h_{(1)}\rightharpoonup_{i} 1_i\# h_{(2)}k\big)  \cdot\,\! \mdl{m}{i}{j}\Big)_{i,j\in \bI_n}.
		\end{align*}
		Denote by $D=\operatorname{diag}(h_{(2)}\rightharpoonup_{j} 1_1,\ldots,h_{(2)}\rightharpoonup_{j} 1_n)\in R$. Then
		\begin{align*}
			\big( h_{(1)}k\rhd  (\mdl{m}{i}{j}) \big)_{i,j\in \bI_n}(h_{(2)}\rhd 1_R)&\overset{\mathclap{\rm(b)}}{=} \big( (1_i\# h_{(1)}k)\cdot\,\!  \mdl{m}{i}{j}\big)_{i,j\in \bI_n} D\\
			&\overset{\mathclap{\rm(b)}}{=} \big( \mdl{m}{i}{j}\cdot\,\! (h_{(1)}k\#1_j)\big)_{i,j\in \bI_n} D\\
			&=  \big(\mdl{m}{i}{j}\cdot (h_{(1)}k\#  h_{(2)}\rightharpoonup_{j} 1_j)\big)_{i,j\in \bI_n}\\  
			&=  \big(\mdl{m}{i}{j}\cdot (h_{[2]}k\#  h_{[1]}\rightharpoonup_{j} 1_j)\big)_{i,j\in \bI_n}\\  
			&=  \big(\mdl{m}{i}{j}\cdot (k\# 1_j)(h\#1_j)\big)_{i,j\in \bI_n}\\  
			&\stackrel{\rm (b)}=  \big( (1_i\# h)(1_i\#k)\cdot \!\mdl{m}{i}{j}\big)_{i,j\in \bI_n} \\
			&=\big( h_{(1)}\rightharpoonup_{i} 1_i)\# h_{(2)}k\cdot\,\! \mdl{m}{i}{j}\big)_{i,j\in \bI_n},     \end{align*}
		which implies {\footnotesize{(LPA3)}}. Thus (ii) $\Rightarrow$ (i) is proved.   
	\end{proof}
	We provide a couple of remarks related to Theorem \ref{teo-pa-gma}.
	\begin{remark} 
		We notice that conditions (b) and (c) of Theorem \ref{teo-pa-gma} are equivalent to (b) and (c'), where
		\[
		{\rm (c')}\qquad mn\cdot (h\# 1_k) =\big( m\cdot (h_{(1)}\# 1_j ) \big)\big( n\cdot (h_{(2)}\# 1_k ) \big), \quad \text{for all }\, h\in H,\,m\in\,\!\! \mdl{M}{i}{j},\,n\in\,\!\! \mdl{M}{j}{k} .
		\]
		Indeed, assume that (b) and (c) are hold. Then
		\begin{align*}
			mn\cdot (h\# 1_k) &=   (1_i\#h)\cdot mn \\
			&=  \big((1_i\#h_{(1)})\cdot m\big)\big((1_j\#h_{(2)})\cdot n\big) \\
			&=  \big( m\cdot (h_{(1)}\# 1_j ) \big)\big( n\cdot (h_{(2)}\# 1_k ) \big).
		\end{align*}
		Analogously, it can be verified that (b) and (c') imply (b) and (c).
	\end{remark}

	\begin{remark}\label{notbimodule}
		Let $H$ be a Hopf algebra acting partially on the generalized matrix algebra $R=(\mdl {M}{i}{j})_{i,j\in \bI_n}$. By Theorem \ref{teo-pa-gma}, ${}_iM_j$ is a left $\underline{R_i \# H}$-module and a right $\underline{H^{op}\# R_j}$-module. However,  ${}_iM_j$  is not necessarily a  $(\underline{R_i \# H} ,\underline{H^{op}\# R_j})$-bimodule. Indeed, since
		$(1_i \# h) \cdot m =m\cdot (h\# 1_j)$, for all $h\in H$ and $m\in {}\mdl {M}{i}{j}$, then we have
		\[
		((1_i \# h) \cdot m)\cdot (k\# 1_j)= (1_i \# k)\cdot ((1_i \# h) \cdot m)=((k_{(1)} \rightharpoonup_{i} 1_i) \# k_{(2)}h) \cdot m ,
		\]
		while 
		\[
		(1_i \# h) \cdot (m\cdot (k\# 1_j))=(1_i \# h)\cdot ((1_i \# k) \cdot m)=((h_{(1)} \rightharpoonup_{i} 1_i) \# h_{(2)}k) \cdot m.
		\]
	\end{remark}
	
	\begin{example}
		Let $A$ be a finite-dimensional algebra, $\{V_1,\ldots,V_n\}$ be  a representative set of  finite-dimensional non-isomorphic simple left $A$-modules, and let  $P_i$ be the projective cover of $V_i$, for all $i=1,\ldots,n$. Then, there exists a primitive idempotent $e_i\in A$ such that $P_i\simeq Ae_{i}$. Consider the idempotent $e=e_1+\ldots+e_n\in A$ and the basic algebra $A^{\operatorname{b}}=eAe$. It is well-known that $A$ and $A^{\operatorname{b}}$ are Morita equivalent. Moreover, the $6$-tuple $(A,A^{\operatorname{b}},Ae,eA,\mu,\nu)$, where $\mu:Ae\otimes_{A^{\operatorname{b}}} eA\to A$ and $\nu:eA\otimes_A Ae\to A^{\operatorname{b}}$ are the multiplication maps, is a strict Morita context. Thus, we have the generalized matrix algebra 
		$$ R = \left( \begin{matrix} A & Ae\\ eA & A^{\rm b} \end{matrix} \right).$$
		In this case, $\mdl{M}{1}{1}=R_1=A$, $\mdl{M}{1}{2}=Ae$, $\mdl{M}{2}{1}=eA$ and $\mdl{M}{2}{2}=R_2=eAe$. 
		Let $H$ be a Hopf algebra and $\cdot:H\otimes A\to A$ a left partial action of $H$ on $A$ such that $h\cdot e=\varepsilon(h)e$, for all $h\in H$. Then
		\[h\cdot (ea)=(h_{(1)}\cdot e)(h_{(2)}\cdot a)=\varepsilon(h_{(1)})e(h_{(2)}\cdot a)=e(h\cdot a),\]
		and similarly $h\cdot (ae)=(h\cdot a)e$, for all $a\in A$ and $h\in H$. Thus, $h\cdot (eae)=e(h\cdot a)e\in eAe$, for all $a\in A$, and we have a partial action of $H$ on $eAe$ obtained by restriction.
		It is straightforward to check that $A$ and $Ae$ are left $\underline{A\#H}$-modules with actions given respectively by 
		\begin{align}
			a\#h\rightharpoonup b&=a(h\cdot b), &a,b\in A,\,\,h\in H,\\
			a\#h\rightharpoonup be&=a(h\cdot b)e, &a,b\in A,\,\,h\in H.
		\end{align}
		On the other hand, $A$ is a right $\underline{H^{\rm op}\#A}$-module and $Ae$ is a right $\underline{H^{\rm op}\#(eAe)}$-module via the actions given respectively by 
		\begin{align}
			b\leftharpoonup h\# a&=(h\cdot b)a, &a,b\in A,\,\,h\in H^{\rm op},\\
			be\leftharpoonup h\# eae&=(h\cdot b)eae,&a,b\in A,\,\,h\in H^{\rm op}.
		\end{align}
		Observe that 
		\[1_A\#h\rightharpoonup be=(h\cdot b)e=be\leftharpoonup h\# e, \qquad 1_A\#h\rightharpoonup b=h\cdot b=b \leftharpoonup h\# 1_A,\quad b\in A,\,\,h\in H.\]
		Hence, the bimodules $\mdl{M}{1}{1}=A$ and $\mdl{M}{1}{2}=Ae$ satisfy the condition (b) of Theorem \ref{teo-pa-gma}. In a similar way, $eA$ and  $eAe$ are left $\underline{(eAe)\#H}$-modules with structures
		\begin{align}
			eae\#h\rightharpoonup eb&= eae(h\cdot b), &a,b\in A,\,\,h\in H,\\
			eae\#h\rightharpoonup ebe&= eae(h\cdot b)e,&a,b\in A,\,\,h\in H.
		\end{align}
		while $eA$ is a right $\underline{\op{H}\# A}$-module and $eAe$ is a right $\underline{\op{H}\# eAe}$-module with structures
		\begin{align}
			eb\leftharpoonup h\# a&=e(h\cdot b)a,&  a,b\in A,\,\,k\in H,\,\,h\in H^{\rm op},\\
			ebe\leftharpoonup h\# eae&=e(h\cdot b)eae,& a,b\in A,\,\,k\in H,\,\,h\in H^{\rm op}.
		\end{align}
		The bimodules $\mdl{M}{2}{1}=eA$ and $\mdl{M}{2}{2}=eAe$ also satisfy the condition (b) of Theorem \ref{teo-pa-gma}. By a direct computation, we can verify that the condition (c) of  Theorem \ref{teo-pa-gma} is valid. Thus, by Theorem \ref{teo-pa-gma}, the linear map $\rhd:H\otimes R\to R$ given by
		\[h\rhd \left( \begin{matrix}
			a & be \\
			ec & ede 
		\end{matrix}\right)=\left( \begin{matrix}
			h\cdot a & (h\cdot b)e \\
			e(h\cdot c) & e(h\cdot d)e 
		\end{matrix}\right)\]
		is a partial action of $H$ on the Morita ring $R$.
	\end{example}

	%\begin{example}Let $G$ be a group and $H=\Bbbk G$ the group algebra....\end{example}
	
	\section{The group case revisited}\label{pgma} In what follows in this section, $\tG$ is a group and $R=(\mdl{M}{i}{j})_{i,j\in \bI_n}$ is a generalized matrix algebra as in \eqref{def-gmr}. We recall  from \cite{CJ} that having a left unital partial action of $G$ on $R$ is the same  as having a  left partial action of the Hopf algebra $\Bbbk G$ on $R$.  
	In Section 4 of \cite{BP}, the authors give sufficient conditions to define a unital partial action of $G$ on $R$. In what follows we relate these conditions with the conditions (a), (b) and (c) in (ii) of Theorem \ref{teo-pa-gma}.
	For the convenience of the reader, we will recall some definitions and results presented in \cite{BP}.
	\smallbreak
	
	\begin{definition}
		\label{def-symetric} Let $n$ be a positive integer, $R=(\mdl{M}{i}{j})_{i,j\in \bI_n}$ be a generalized matrix algebra, and $\mcI=\{I_j\subset R_j\,:\,I_j \text{ is an ideal of }R_j,\text{ for all }j\in \bI_n\}$. The $(R_i,R_j)$-bimodule $\mdl{M}{i}{j}$ will be called $\mcI$-symmetric if $I_i\cdot\,\!\mdl{M}{i}{j}=\!\mdl{M}{i}{j}\cdot I_j$. When $\mdl{M}{i}{j}$ is $\mcI$-symmetric for all $i,j\in \bI_n$, we say that $R$ is $\mcI$-symmetric.
	\end{definition}
	
	The next Lemma is Corollary 3.8 of \cite{BP}.
	
	\begin{lemma}
		\label{cor:ideals}  Let $n$ be a positive integer,  $R=(\mdl{M}{i}{j})_{i,j\in \bI_n}$ be a generalized matrix algebra and $\mcI=\{I_j\subset R_j\,:\,I_j \text{ is an ideal of }R_j,\text{ for all }j\in \bI_n\}$ be a finite family of ideals. If $R$ is $\mcI$-symmetric then 
		\begin{align*}\label{expl-ideal}
			I = \left( \begin{matrix} I_{11} & I_{12} & \ldots & I_{1n}\\ 
				I_{21} & I_{22} & \ldots &  I_{2n}  \\ 
				\vdots & \vdots & \vdots &\vdots \\ 
				I_{n1} &I_{n2} &\ldots  & I_{nn}\end{matrix} \right), \quad \text{where }\, I_{jk}:=I_j\cdot{}\mdl{M}{j}{k}+{}\mdl{M}{j}{k}\cdot I_k,
		\end{align*} 
		is an ideal of $R$. 
	\end{lemma}

	Assume that, for each $i\in \bI_n$, there is a partial action  $\af^{(i)}=\big(D^{(i)}_g,\af^{(i)}_g\big)_{g\in \tG}$
	of $\tG$ on $R_i$. Consider $\mcI_g:=\big\{D^{(i)}_g\,:\,i\in\bI_n\big\}$ and suppose that $R$ is $\mcI_g$-symmetric, for all $g\in \tG$. By Lemma  \ref{cor:ideals}, the subset $I_g$ of $R$ given by 
	\begin{equation}\label{def-ideal}
		I_{g}=\big(D^{(i)}_g\,\mdl{M}{i}{j}\big)_{i,j\in \bI_n},
	\end{equation} is an ideal of $R$.
	\medbreak
	
	Now we recall the notion of datum for $R$  introduced in Definition 4.2 of \cite{BP}.
	
	\begin{definition}
		Let $i\in \bI_n$ and $g\in \tG$.  Suppose that $\af^{(i)}$ and $\mcI_g$ are as above and that $R$ is $\mcI_g$-symmetric. For each $g\in \tG$, assume that there exists a collection   
		of linear bijections \[\gamma_g=\left\{\gamma^{(ij)}_g:D^{(i)}_{g\m}\,\mdl{M}{i}{j}\to D^{(i)}_g\,\mdl{M}{i}{j}\right\}_{i,j\in \bI_n}.\]  The pair 
		$\mcD=\big(\{\alpha^{(i)}\}_{i\in \bI_n}, \{\gamma_{g}\}_{g\in\tG}\big)$
		is called a {\it datum for $R$} if the following statements hold: for all $g,h\in \tG$, $i,j,k\in \bI_n$, \vspace{.1cm}
		\begin{align}
			\label{cond1} &\gamma^{(ii)}_g=\alpha^{(i)}_g,\quad\quad\gamma^{(ij)}_e=\id_{M_{ij}},
			\\[.2em]	
			\label{cond2} &\gamma^{(ik)}_g(u)\gamma^{(kj)}_g(v)=\gamma^{(ij)}_g(uv),\,\,\,\text{for all } u\in D^{(i)}_{g\m}\,\mdl{M}{i}{k},\,v\in D^{(k)}_{g\m}\,\mdl{M}{k}{j},\\[.2em]
			\label{cond03} &\big(\gamma^{(ij)}_h\big)^{\m}\!\!\left(D^{(i)}_{g\m}\,\mdl{M}{i}{j}\cap D^{(i)}_{h}\,\mdl{M}{i}{j}\right) \subset D^{(i)}_{(gh)\m}\,\mdl{M}{i}{j}, \\[.4em]
			\label{cond4} &\gamma^{(ij)}_g\left(\gamma^{(ij)}_h(a)\right)=\gamma^{(ij)}_{gh}(a),\,\,\,\text{for all }  a\in \big(\gamma^{(ij)}_h\big)^{\m}\!\!\left(D^{(i)}_{g\m}\,\mdl{M}{i}{j}\cap D^{(i)}_{h}\,\mdl{M}{i}{j}\right). 
		\end{align} 
		
	\end{definition}

	It was proved in Theorem 4.3 of \cite{BP} that there is a partial action of $\tG$ on $R$ associated with each datum for $R$.
	\medbreak
	
	\begin{remark}\label{resu}\label{r1} 
		{\rm  Let $\mcD=\big(\{\alpha^{(i)}\}_{i\in \bI_n}, \{\gamma_{g}\}_{g\in\tG}\big)$ be a datum for $R$. By Remark 4.1 of  \cite{BP}, for all $g,h\in \tG$ and $i,j\in \bI_n$, we have that \vspace{.1cm}
			\begin{enumerate}[\rm (i)]
				\item $(\gamma^{(ij)}_g)^{\m}=\gamma^{(ij)}_{g\m}$; \vspace{.2cm}
				\item $\big(\gamma^{(ij)}_h\big)^{\m}\!\!\left(D^{(i)}_{g\m}\,\mdl{M}{i}{j}\cap D^{(i)}_{h}\,\mdl{M}{i}{j}\right)=D^{(i)}_{(gh)\m}\,\mdl{M}{i}{j}\cap D^{(i)}_{h\m}\,\mdl{M}{i}{j}$;\vspace{.2cm}
				%\item $D^{(i)}_{g}\!\,\mdl{M}{i}{j}=1^{(i)}_g\!\,\mdl{M}{i}{j}=\!\,\mdl{M}{i}{j}1^{(j)}_g=\!\,\mdl{M}{i}{j}D^{(j)}_{g}$.
		\end{enumerate}}
	\end{remark}

	\begin{remark} 
		Let $\mcD=\big(\{\alpha^{(i)}\}_{i\in \bI_n}, \{\gamma_{g}\}_{g\in\tG}\big)$ be a datum for $R$. Assume that $\alpha^{(i)}$ is unital, for all $i\in \bI_n$, and denote by $1_g^{(i)}$ the identity element of $D_{g}^{(i)}$, $g\in \tG$.   For each $g\in \tG$, consider the diagonal matrix 
		\begin{align}\label{unity-for-gamma}
			1_g:=\operatorname{diag}\big(1^{(1)}_g,\ldots,1^{(n)}_g\big)=  \left( \begin{matrix} 1^{(1)}_{g} & 0 & \ldots & 0 \\ 
				0 & 1^{(2)}_{g} & \ddots & \vdots  \\ 
				\vdots & \ddots & \ddots & 0 \\ 
				0 &\ldots & 0 & 1^{(n)}_{g} \end{matrix} \right) \in R.
		\end{align} It was proved in Lemma 4.5 of \cite{BP} that  $1_g$ is a central idempotent of $R$ if and only if
		\begin{equation}\label{cent}
			1^{(i)}_gm=m1^{(j)}_g, \quad \text{for all } m\in \!\,\mdl{M}{i}{j}\text{ and } i,j\in \bI_n.
		\end{equation}  
		In this case, by Lemma 4.5 of \cite{BP}, we have that $I_g=R1_g$.
		
	\end{remark}

	\begin{definition}\label{unital-datum}
		A datum $\mcD=\big(\{\alpha^{(i)}\}_{i\in \bI_n}, \{\gamma_{g}\}_{g\in\tG}\big)$ for $R$ is called unital, if the ideal of $R$ given in \eqref{def-ideal} is unital with identity element $1_g$ given by \eqref{unity-for-gamma}, for all $g\in \tG$.
	\end{definition}
	
	In order to relate the conditions of a datum for $R$ with the conditions given in (ii) of Theorem \ref{teo-pa-gma}, we consider the following statement:\vspace{.2cm}
	\begin{enumerate} [$\quad$\rm (a)']
		\item  there is a  left partial action $\rightharpoonup_{i}: \Bbbk\tG\otimes R_i\to R_i$ of $\Bbbk\tG$ on $R_i$ such that $g\rightharpoonup_{i} 1_{R_i}$ satisfies \eqref{cent}, for all $g\in \tG$ and $i\in \bI_n$.
	\end{enumerate}
	%Observe that (a)' is more restrictive than (a) of Theorem \ref{teo-pa-gma} (ii).
	\begin{proposition}\label{htod}
		{\rm  Let $\mcD=\big(\{\alpha^{(i)}\}_{i\in \bI_n}, \{\gamma_{g}\}_{g\in\tG}\big)$ be a unital datum for $R$. Then the Hopf algebra $H=\Bbbk\tG$  satisfies the conditions (a)' and (b)-(c) of Theorem \ref{teo-pa-gma}  (ii).}
	\end{proposition}
	\begin{proof}
		
		\noindent Let $i\in \bI_n$. Consider the linear map $\rightharpoonup_{i}\,: \Bbbk\tG\otimes R_i\to R_i$ given by
		\[g\rightharpoonup_{i}x=\gamma_g^{(ii)}(x1^{(i)}_{g\m}), \quad g\in \tG,\,\,x\in R_i.\] Since $\gamma_g^{(ii)}=\alpha_g^{(i)}$, it follows that $\rightharpoonup_{i}$ is a  left partial action of $\Bbbk \tG$ on $R_i$ and $g\rightharpoonup_{i} 1_{R_i}=\alpha_g^{(i)}(1^{(i)}_{g\m})=1_g^{(i)}$. Hence (a)' is proved.
		To prove (b) of Theorem \ref{teo-pa-gma}, we recall from \eqref{lsm}  that  $\underline{}R_i\# \Bbbk\tG\simeq R_i\rtimes_{\alpha_i} \tG$  as algebras. We define the  linear map $\cdot:(R_i\rtimes_{\alpha_i} \tG)\otimes\,\! \mdl{M}{i}{j}\to\!\, \mdl{M}{i}{j}$  by
		\[(r\delta_g)\cdot m=r\gamma_g^{(ij)}(1^{(i)}_{g\m}m)=r\gamma_g^{(ij)}(m1^{(j)}_{g\m}), \quad r\in D_g^{(i)}, \,g\in \tG, \,m\in\!\, \mdl{M}{i}{j}.\]
		We claim that this linear map defines a left $R_i\rtimes_{\alpha_i} \tG$-module structure on $\mdl{M}{i}{j}$. Using that $\gamma^{(ij)}_e=\id_{M_{ij}}$ it is immediate that $(1_i\delta_e)\cdot m=m$, for all $m\in\!\, \mdl{M}{i}{j}$. Moreover, for all $g,h\in \tG$, 
		$r\in D_g^{(i)}$ and $s\in D_h^{(i)}$ we have
		\begin{align*}
			\big((r\delta_g)(s\delta_h)\big)\cdot m&=\,r\alpha^{(i)}_g(s1^{(i)}_{g\m})\delta_{gh}\cdot m\\
			&=\,r\alpha^{(i)}_g(s1^{(i)}_{g\m})\gamma_{gh}^{(ij)}(m1^{(j)}_{(gh)\m})\\
			&=\,r\alpha^{(i)}_g(s1^{(i)}_{g\m})1_g^{(i)}\gamma_{gh}^{(ij)}(m1^{(j)}_{(gh)\m})\\
			&\overset{\mathclap{\eqref{cent}}}{=}\,r\alpha^{(i)}_g(s1^{(i)}_{g\m})\gamma_{gh}^{(ij)}(m1^{(j)}_{(gh)\m})1_g^{(j)} \\
			&=\,r\alpha^{(i)}_g(s1^{(i)}_{g\m})\gamma_{gh}^{(ij)}(m1^{(j)}_{(gh)\m})1_g^{(j)}1_{gh}^{(j)} \\
			&\overset{\mathclap{\eqref{cond1}}}{=}\, r\alpha^{(i)}_g(s1^{(i)}_{g\m})\gamma_{gh}^{(ij)}(m1^{(j)}_{(gh)\m})\gamma_{gh}^{(jj)}(  1_{h\m}^{(j)}1_{(gh)\m}^{(j)})\\
			&\overset{\mathclap{\eqref{cond2}}}{=}\,r\alpha^{(i)}_g(s1^{(i)}_{g\m})\gamma_{gh}^{(ij)}(m1^{(j)}_{(gh)\m}1_{h\m}^{(j)})\\
			&\overset{\mathclap{\eqref{cond4}}}{=}\, r\alpha^{(i)}_g(s1^{(i)}_{g\m})\gamma_{g}^{(ij)}(\gamma_{h}^{(ij)}(m1^{(j)}_{(gh)\m}1_{h\m}^{(j)}))\\
			&=  r\alpha^{(i)}_g(s1^{(i)}_{g\m})\gamma_g^{(ij)}(\gamma_h^{(ij)}(m1^{(j)}_{h\m}  )1^{(j)}_{g\m})\\
			&=\,r\gamma^{(ii)}_g(s1^{(i)}_{g\m})\gamma_g^{(ij)}(\gamma_h^{(ij)}(m1^{(j)}_{h\m}  )1^{(j)}_{g\m}) \\
			&= \, r\gamma_g^{(ij)}(s\gamma_h^{(ij)}(m1^{(j)}_{h\m}  )1^{(j)}_{g\m})   \\
			&=\,(r\delta_g)\cdot((s\delta_h)\cdot m).    
		\end{align*}
		
		To complete the proof of (b), we need to verify that $\mdl{M}{i}{j}$ is a right $\underline{\op{\Bbbk\tG}\# {R_j}}$-module. By \eqref{rsm}, $\underline{\op{\Bbbk\tG}\# {R_j}}\simeq  \op{(\op{R_j}\ltimes_{\alpha_j} \tG)}$ as algebras. Thus, we will prove that $\mdl{M}{i}{j}$ is a left   $(\op{R_j}\rtimes_{\alpha_j} \tG)$-module. For that, define the linear map $\cdot:\op{R_j}\rtimes_{\alpha_j}\! \tG\otimes\,\! \mdl{M}{i}{j}\to\!\, \mdl{M}{i}{j}$  by
		\[(r\delta_g)\cdot m=\gamma_g^{(ij)}(m1^{(j)}_{g\m})r,\quad r\in D_g^{(j)}, \,g\in \tG, \,m\in\!\, \mdl{M}{i}{j}.\]
		Notice that by \ref{cent} this action is well defined. Also, by \eqref{cond1}, we have $(1_{R_j}\delta_e)\cdot m=\gamma_e^{(ij)}(m1^{(j)}_{e})1_{R_j}=\gamma_e^{(ij)}(m1_{R_j})1_{R_j}=m$. Finally, observe that
		\begin{align*}
			\big( (r\delta_g)(s\delta_h)\big)\cdot m&=\alpha^{(j)}_g(s1^{(j)}_{g\m})r\delta_{gh}\cdot m\\
			&=\gamma_{gh}^{(ij)}(m1^{(j)}_{(gh)\m})\alpha^{(j)}_g(s1^{(j)}_{g\m})r\\
			&=\gamma_{gh}^{(ij)}(m1^{(j)}_{(gh)\m}) 1_g^{(j)}1_{gh}^{(j)}\alpha^{(j)}_g(s1^{(j)}_{g\m})r\\
			&=\gamma_g^{(ij)}(\gamma_h^{(ij)}(m1^{(j)}_{h\m}  )1^{(j)}_{g\m})\alpha^{(j)}_g(s1^{(j)}_{g\m})r \\ 
			&=\gamma_g^{(ij)}(\gamma_h^{(ij)}(m1^{(j)}_{h\m}  )s1^{(j)}_{g\m})r, \end{align*}
		and 
		\begin{align*}
			(r\delta_g)\cdot \big( (s\delta_h)\cdot m\big)&=\, (r\delta_g)\cdot   \big(\gamma_h^{(ij)}(m1^{(j)}_{h\m})s\big)\\[.3em]
			&=\,\gamma_g^{(ij)}\Big( \big(\gamma_h^{(ij)}(m1^{(j)}_{h\m})s\big) 1^{(j)}_{g\m}\Big)r\\[.3em]
			&=\,\gamma_g^{(ij)}\big( \gamma_h^{(ij)}(m1^{(j)}_{h\m})s 1^{(j)}_{g\m}\big)r.\\
		\end{align*}
		
		Given $h\in \tG$ and $m\in\!\, \mdl{M}{i}{j}$, from \eqref{lsm} follows that
		$$(1_i\#h)\cdot m=(1_h^{(i)}\delta_h)\cdot m=\gamma_h^{(ij)}(m1^{(j)}_{h\m}),$$
		and  \eqref{rsm} implies
		$$m\cdot (h\# 1_j)=(1^{(j)}_{h}\delta_{h})\cdot m=\gamma_{h}^{(ij)}(m1^{(j)}_{h\m}).$$
		Thus, (b) is proved.\vspace{.1cm}
		
		Now, we will see that (c) holds. In fact, if $n\in\!\, \mdl{M}{j}{k}$ then
		\begin{align*}
			(1_i\#h)\cdot mn&=\,\,\gamma_h^{(ik)}(mn1_{h\m}^{(k)})\\
			&\overset{\mathclap{\eqref{cent}}}{=}\,\,\gamma_h^{(ik)}(m1_{h\m}^{(j)}n1_{h\m}^{(k)})\\
			&\overset{\mathclap{\eqref{cond2}}}{=}\,\,\gamma_h^{(ij)}(m1_{h\m}^{(j)})\gamma_h^{(jk)}(n1_{h\m}^{(k)})\\
			&=\,\,\big((1_i\#h)\cdot m\big)\big((1_j\#h)\cdot n\big),
		\end{align*}
		which finishes the proof.
	\end{proof}
	We proceed with the next.
	\begin{proposition}	\label{ga2} {\rm
			Assume that $\Bbbk\tG$ satisfies the conditions (a)' and (b)-(c) of Theorem \ref{teo-pa-gma} (ii). Then:
			\begin{itemize}
				\item [$\circ$] given $i\in \bI_n$, there exists a unital partial action  $\af^{(i)}=\big(D^{(i)}_g,\af^{(i)}_g\big)_{g\in \tG}$ of $\tG$ on $R_i$,\vspace{.1cm}
				\item [$\circ$]  there exists a family of $\Bbbk$-linear bijections $\gamma_g=\left\{\gamma^{(ij)}_g:D^{(i)}_{g\m}\,\mdl{M}{i}{j}\to D^{(i)}_g\,\mdl{M}{i}{j}\right\}_{i,j\in \bI_n}$, for all $g\in \tG$,
			\end{itemize}\vspace{.1cm}
			such that $\mcD=\big(\{\alpha^{(i)}\}_{i\in \bI_n}, \{\gamma_{g}\}_{g\in\tG}\big)$ is a unital datum for $R$.}
	\end{proposition}
	
	\begin{proof}
		From Example \ref{ex-pa-group} we obtain that $\alpha^{(i)}=\big(D^{(i)}_g,\af^{(i)}_g\big)_{g\in \tG}$ is a unital partial action of $\tG$ on $R_i$, where $$D^{(i)}_g=R_i1^{(i)}_g, \qquad 1^{(i)}_g:=g\rightharpoonup_{i} 1_{R_i},\qquad\af^{(i)}_g(y)=g\rightharpoonup_{i} y,\quad y\in D^{(i)}_{g\m}.$$
		Notice that  \eqref{cent} implies that $D^{(i)}_{g}\,\mdl{M}{i}{j}=\,\!\mdl{M}{i}{j}\,D^{(j)}_{g}$, for all $i,j\in \bI_n$ and $g\in\tG$. Hence $R$ is $\mcI_g$-symmetric, where $\mcI_g=\big\{D^{(i)}_g\,:\,i\in\bI_n\big\}$. Also, by (b) in (ii) of Theorem \ref{teo-pa-gma}, for each $(i,j)\in \bI_n\times \bI_n$, we have that $\mdl{M}{i}{j}$ is a left 
		$R_i\rtimes_{\alpha_i} \tG$-module. Using this left action we define the following $\Bbbk$-linear morphisms   
		\begin{equation}\label{newg}  
			\gamma^{(ij)}_g: D^{(i)}_{g\m}\,\mdl{M}{i}{j}\to D^{(i)}_g\,\mdl{M}{i}{j},\quad\gamma_g^{(ij)}(x)= (1^{(i)}_g\delta_g)\cdot x,\quad x\in  D^{(i)}_{g\m}\,\mdl{M}{i}{j}. \end{equation}
		Notice that $\gamma^{(ij)}_g$ is well-defined. In fact, it follows from (b) of Theorem \ref{teo-pa-gma} (ii)  that $r_i\delta_e\cdot m=r_i\cdot m$, for all $r_i\in R_i$ and $m\in\,\! \mdl{M}{i}{j}$. Hence, given $x\in  D^{(i)}_{g\m}\,\mdl{M}{i}{j}$,
		\[(1^{(i)}_g\delta_g)\cdot x= ((1^{(i)}_g\delta_e)\cdot (1^{(i)}_g\delta_g))\cdot x= (1^{(i)}_g\delta_e)\cdot [\big(1^{(i)}_g\delta_g)\cdot x]=  1^{(i)}_g\cdot \big( (1^{(i)}_g\delta_g)\cdot x\big),\]
		which implies that $(1^{(i)}_g\delta_g)\cdot x\in D^{(i)}_g\,\mdl{M}{i}{j}$.
		Also,  $$\gamma^{(ij)}_{g\m}\gamma^{(ij)}_g(x)=1^{(i)}_{g\m}\delta_e\cdot x=1^{(i)}_{g\m}\cdot x=x.$$ 
		Similarly, $\gamma^{(ij)}_g\gamma^{(ij)}_{g\m}=\id_{D^{(i)}_{g}\,\mdl{M}{i}{j}} $ and consequently $\gamma^{(ij)}_{g}$ is a $\Bbbk$-linear isomorphism with inverse $\gamma^{(ij)}_{g\m}$.
		Clearly,  $\gamma_e^{(ij)}={\rm id}_{\mdl{M}{i}{j}}$. Moreover, if $\Phi_{ii}$, $\psi_{ii}$ and $\pi_{ii}$ are the maps defined in the proof of (i) $\Rightarrow$ (ii) of Theorem \ref{teo-pa-gma}, then
		\begin{align*} 
			\gamma^{(ii)}_g(a)=(1_g\delta_g)\cdot a=\Phi_{ii}(1_g\#g)(a)=(\psi_{ii}(1_g)\circ \pi_{ii}(g))(a)
			=\alpha_g(a),
		\end{align*} 
		for all $a\in D^{(i)}_{g\m}$. Hence, \eqref{cond1} is proved. In order to prove \eqref{cond2}, consider $u\in D^{(i)}_{g\m}\,\mdl{M}{i}{j}$ and $v\in D^{(j)}_{g\m}\,\mdl{M}{j}{k}$. We recall that $1_i\#g=1_i(g\rightharpoonup_{i} 1_i)\otimes g=1^{(i)}_g\#g$. Then, it follows from (c) of Theorem \ref{teo-pa-gma} (ii) that 
		\begin{align*}
			\gamma^{(ij)}_g(u)\gamma^{(jk)}_g(v)&=\big((1^{(i)}_g\delta_g)\cdot u\big)\big((1^{(j)}_g\delta_g)\cdot v\big)=\big((1^{(i)}_g\# g)\cdot u\big)\big((1^{(j)}_g\# g)\cdot v\big)\\
			&=\big((1_i\# g)\cdot u\big)\big(1_j\# g)\cdot v\big)=(1_i\# g)\cdot (uv)\\
			&=(1^{(i)}_g\# g)\cdot (uv)=(1^{(i)}_g\delta_g)\cdot (uv)=\gamma^{(ik)}_g(uv).
		\end{align*}
		Now, we  prove  \eqref{cond03}. Consider  $ x\in (\gamma^{(ij)}_h\big)^{\m}\!\!\left(D^{(i)}_{g\m}\,\mdl{M}{i}{j}\cap D^{(i)}_{h}\,\mdl{M}{i}{j}\right)$. Then, observe that $y=\gamma^{(ij)}_{h}(x)\in D^{(i)}_{g\m}\,\mdl{M}{i}{j}\cap D^{(i)}_{h}\,\mdl{M}{i}{j}$. Since $y=1^{(i)}_{g\m}1^{(i)}_hy$, it follows that
		\begin{align*} 
			x&=\gamma^{(ij)}_{h\m}(y)= \gamma^{(ij)}_{h\m}(1^{(i)}_{g\m}1^{(i)}_hy)=\alpha_{h\m}(1^{(i)}_{g\m}1^{(i)}_h)\gamma^{(ij)}_{h\m}(y)=1^{(i)}_{(gh)\m}1^{(i)}_{h\m}\gamma^{(ij)}_{h\m}(y),
		\end{align*} 
		and hence $x\in D^{(i)}_{h\m}\,\mdl{M}{i}{j}\cap D^{(i)}_{(gh)\m}\,\mdl{M}{i}{j}\subseteq D^{(i)}_{(gh)\m}\,\mdl{M}{i}{j}$. Finally, we prove  \eqref{cond4}. Let  $a\in \gamma^{(ij)}_{h^{\m}}( D^{(i)}_{g\m}\,\mdl{M}{i}{j}\cap D^{(i)}_{h}\,\mdl{M}{i}{j})$. Then  
		\begin{align*}
			\gamma^{(ij)}_g\left(\gamma^{(ij)}_h(a)\right)&=(1^{(i)}_g\delta_g)\cdot\big((1^{(i)}_h\delta_h)\cdot a\big)=(1^{(i)}_g1^{(i)}_{gh}\delta_{gh})\cdot a=1^{(i)}_g\gamma^{(ij)}_{gh}(a).\end{align*}
		However, using that $a\in D^{(i)}_{h\m}\,\mdl{M}{i}{j}$, we obtain
		\begin{align*}
			\gamma^{(ij)}_{gh}(a)&= \gamma^{(ij)}_{gh}( 1^{(i)}_{h\m}a)=\gamma^{(ii)}_{gh}(1^{(i)}_{h\m}1^{(i)}_{(gh)\m})\gamma^{(ij)}_{gh}(a) \\ &=\alpha^{(i)}_{gh}(1^{(i)}_{h\m}1^{(i)}_{(gh)\m})\gamma^{(ij)}_{gh}(a)=1^{(i)}_g\gamma^{(ij)}_{gh}(a).     
		\end{align*}
		Thus $\gamma^{(ij)}_g\left(\gamma^{(ij)}_h(a)\right)=\gamma^{(ij)}_{gh}(a)$. Hence $\mcD=\big(\{\alpha^{(i)}\}_{i\in \bI_n}, \{\gamma_{g}\}_{g\in\tG}\big)$ is a unital datum for $R$.
	\end{proof}
	We finish this section with the following result.
	
	\begin{theorem}\label{teo-group-case}{\rm
			The following assertions are equivalent:
			\begin{enumerate}[\rm (i)]
				\item there is a unital datum $\mcD=\big(\{\alpha^{(i)}\}_{i\in \bI_n}, \{\gamma_{g}\}_{g\in\tG}\big)$ for $R$.\vspace{.2cm}
				\item The Hopf algebra $\Bbbk\tG$ satisfies the conditions (a)' and (b)-(c)  of Theorem \ref{teo-pa-gma} (ii).
		\end{enumerate}}
	\end{theorem}
	\begin{proof}
		It follows directly from Propositions \ref{htod} and \ref{ga2}.
	\end{proof}
	
	\section{Morita equivalent partial Hopf actions}
	
	We start by recalling from \cite{AF} the notion of  Morita equivalence for partial actions of Hopf algebras.
	
	\begin{definition} \label{moritaequivalent} (\cite{AF}, Definition 69)
		Let $H$ be a Hopf algebra and $A$ and $B$ two partial $H$-module algebras with partial actions $\rightharpoonup_{A} $ and $\rightharpoonup_{B} $, respectively. We say that these two partial actions are Morita equivalent if
		\begin{enumerate}[\rm (i)]
			\item there exists a strict Morita context $(A,B, {}_AM_B ,{}_BN_A , \tau, \sigma)$ between $A$ and $B$, \vspace{.1cm}
			\item there exists a partial action of $H$ on the Morita ring $\left( \begin{matrix} A & M\\ N & B \end{matrix} \right)$ such that the restrictions to $\left( \begin{matrix} A & 0\\ 0 & 0 \end{matrix} \right)$ and to $\left( \begin{matrix} 0 & 0\\ 0 & B \end{matrix} \right)$ coincide with $\rightharpoonup_{A} $ and $\rightharpoonup_{B}$, respectively.
		\end{enumerate}
	\end{definition}
	
	The next result characterizes Morita equivalent Hopf partial actions.
	
	\begin{proposition}\label{prop-morita-eq}
		Let $H$ be a Hopf algebra and $A$ and $B$ two partial $H$-module algebras with partial actions $\rightharpoonup_{A} $ and $\rightharpoonup_{B} $, respectively. The following statements are equivalent:
		\begin{enumerate}[\rm (1)]
			\item  $\rightharpoonup_{A} $ and $\rightharpoonup_{B} $ are Morita equivalent,\vspace{.1cm}
			\item  The following assertions hold:\vspace{.1cm}
			\begin{enumerate}[\rm (i)]
				\item there exits a strict Morita context $(A,B,M,N,\mu,\nu)$, \vspace{.1cm}
				\item $M$ is a left $\underline{A\# H}$-module and a right $\underline{\op{H}\# B}$-module such that 
				\begin{align*}
					& (1_A\#h)\cdot m=m\cdot (h\# 1_B),& &\,\,(1_A\#h)\cdot a= h\rightharpoonup_{A} a,&\\[.2em]
					&\,\,\,(a'\# 1_H)\cdot m=a'm,& &m\cdot (1_H\#b)=mb,&
				\end{align*}
				for all $h\in H$, $m\in M$, $a,a'\in A$, $b\in B$, \vspace{.1cm}
				\item $N$ is a left $\underline{B\# H}$-module and a right $\underline{\op{H}\# A}$-module such that 
				\begin{align*}
					& (1_B\#h)\cdot n=n\cdot (h\# 1_A),& &\,\,(1_B\#h)\cdot b= h\rightharpoonup_{B} b,&\\[.2em]
					&\,\,\,(b'\# 1_H)\cdot n=b'n,& &n\cdot (1_H\#a)=na,&
				\end{align*}
				for all $h\in H$, $n\in N$, $b,b'\in B$, $a\in A$.
			\end{enumerate}\vspace{.1cm}
		\end{enumerate}     
		
	\end{proposition}
	
	\begin{proof}
		It follows from  applying Theorem \ref{teo-pa-gma}  to the Morita ring $ R = \left( \begin{matrix} A & M \\ N & B \end{matrix} \right).$
	\end{proof}

	Now we recall the following notion.  
	
	\begin{definition}(\cite{AF}, Definition 72)
		Let $H$ be a Hopf algebra and $B$ be a right partial $H$-module algebra. A vector space $N$ is a right partial $(B,H)$-module if $N$ is a right $B$-module together with a linear map
		$N\otimes H\to N$, $n\otimes h\mapsto nh$ such that
		\begin{enumerate}[\rm (i)]
			\item $n1_H=n$,\vspace{.1cm}
			\item $((nk)b)h=(n(kh_{(1)}))(b\cdot h_{(2)})$,
		\end{enumerate}
		for all $h,k\in H$, $b\in B$ and $n\in N$.
	\end{definition}

	We conclude this work by highlighting an inconsistency in the statement of Proposition 73 of \cite{AF}. In fact, the authors assume that $H$ partially acts on $B$ from the left. However, since $N$ is a right $(B,H)$-module, $H$ indeed acts partially on $B$ from the right.

	%As we have seen in Theorem \ref{teo-pa-gma}, condition (2) of the definition of Morita equivalent partial actions is equivalent to impose that $M$ is a left $\underline{A\# H}$ and a right $\underline{H^{op}\# B}$ module, while $N$ is a left $\underline{B\# H}$ and a right $\underline{H^{op}\# A}$-module, but they are not bimodules, as emphasized before in Remark \ref{notbimodule}.  %But observe that $M$ is a right $\underline{H^{op}\# A}$-module, if and only if, it is a right  $\underline{A\# H}$-module.++++++++++++++++
	
	%Therefore, we must draw attention to an error made in the reference \cite{AF}. 

	%which was supposed to be a covariance condition equivalent to our module structures over the partial smash products.  But in item (4) of that Definition, the authors impose exactly the bimodule condition, which does not actually work, as explained above. This does not invalidate the Definition \ref{moritaequivalent}, of Morita equivalent partial actions of a Hopf algebra $H$, only the result where the authors describe what should be an equivalent condition for two partial actions to obey, in order to make them Morita equivalent partial actions.

\end{document}